\documentclass{amsart}

\usepackage[utf8]{inputenc}
\usepackage{hyperref}
\hypersetup{
	colorlinks,
	citecolor=black,
	filecolor=black,
	linkcolor=black,
	urlcolor=black,
	linktoc=all
}
\hypersetup{colorlinks=true}
\usepackage{breakurl}

\usepackage{libertine}
\usepackage{libertinust1math}

\usepackage{amsmath}
\usepackage{amsfonts}
\usepackage{mathtools}

\usepackage{picinpar,moresize,xfrac,graphpap,dcolumn,wrapfig,graphicx}

\usepackage{subcaption}
\usepackage{stackengine}

\usepackage{tikz,pgfplots}
\usepackage{tikz-cd}
\usepackage{pgf}
\usepgfplotslibrary{colormaps}
\usetikzlibrary{pgfplots.colormaps}
\usetikzlibrary{arrows,automata,positioning,backgrounds}
\usepackage{rotating}
\pgfplotsset{compat=newest}
\pgfplotsset{plot coordinates/math parser=false}
\newlength\figureheight
\newlength\figurewidth

\usepackage{soul,array,calc,url,ragged2e,graphpap}
\usepackage{booktabs} 
\usepackage{tabularx}
\usepackage{colortbl}
\usepackage{multirow}
\usepackage{hyphenat}
\usepackage{transparent}
\usepackage{mathrsfs}
\usepackage{mathtools}
\mathtoolsset{centercolon}
\usepackage{nicefrac}
\usepackage{units}
\usepackage[normalem]{ulem}
\usepackage{cancel}
\usepackage{pbox}
\usepackage{multicol}
\usepackage{cases}

\usepackage{cleveref}
\crefmultiformat{subequation}%
{\edef\crefstripprefixinfo{#1}(#2#1#3}%
{,#2\crefstripprefix{\crefstripprefixinfo}{#1}#3)}%
{,#2\crefstripprefix{\crefstripprefixinfo}{#1}#3}%
{,#2\crefstripprefix{\crefstripprefixinfo}{#1}#3)}

\makeatletter
\DeclareFontFamily{OMX}{MnSymbolE}{}
\DeclareSymbolFont{MnLargeSymbols}{OMX}{MnSymbolE}{m}{n}
\SetSymbolFont{MnLargeSymbols}{bold}{OMX}{MnSymbolE}{b}{n}
\DeclareFontShape{OMX}{MnSymbolE}{m}{n}{
    <-6>  MnSymbolE5
   <6-7>  MnSymbolE6
   <7-8>  MnSymbolE7
   <8-9>  MnSymbolE8
   <9-10> MnSymbolE9
  <10-12> MnSymbolE10
  <12->   MnSymbolE12
}{}
\DeclareFontShape{OMX}{MnSymbolE}{b}{n}{
    <-6>  MnSymbolE-Bold5
   <6-7>  MnSymbolE-Bold6
   <7-8>  MnSymbolE-Bold7
   <8-9>  MnSymbolE-Bold8
   <9-10> MnSymbolE-Bold9
  <10-12> MnSymbolE-Bold10
  <12->   MnSymbolE-Bold12
}{}
\let\llangle\@undefined
\let\rrangle\@undefined
\DeclareMathDelimiter{\llangle}{\mathopen}%
                     {MnLargeSymbols}{'164}{MnLargeSymbols}{'164}
\DeclareMathDelimiter{\rrangle}{\mathclose}%
                     {MnLargeSymbols}{'171}{MnLargeSymbols}{'171}
\makeatother

\newcommand{\figref}[1]{\textup{Fig.~\ref{#1}}}

\newcommand{\teqref}[1]{\textup{Eq.~(\ref{#1})}}

\newcommand{\secref}[1]{\textup{Section~\ref{#1}}}
\newcommand{\appref}[1]{\textup{Appendix~\ref{#1}}}

\newcommand{\thmref}[1]{\textup{Theorem~\ref{#1}}}
\newcommand{\lemref}[1]{\textup{Lemma~\ref{#1}}}

\newcommand{\defref}[1]{\textup{Definition~\ref{#1}}}

\newcommand{\corref}[1]{\textup{Corollary~\ref{#1}}}

\def\etal{\emph{et al.}}

\def\ie{\emph{i.e.}}
\def\eg{\emph{e.g.}}

\def\resp{resp.}

\def\RR{\mathbb{R}}
\def\SS{\mathbb{S}}

\def\ZZ{\mathbb{Z}}

\def\cB{\mathcal{B}}

\def\cF{\mathcal{F}}
\def\cG{\mathcal{G}}

\def\cL{\mathcal{L}}

\def\cT{\mathcal{T}}

\def\ba{\mathbf{a}}
\def\bb{\mathbf{b}}

\usepackage{bbm}
\DeclareSymbolFont{bbold}{U}{bbold}{m}{n}
\DeclareSymbolFontAlphabet{\mathbbold}{bbold}

\def\det{\operatorname{det}}

\def\id{\operatorname{id}}

\DeclareMathOperator*{\argmin}{argmin}

\usepackage[backend=biber, style=numeric, natbib=true, maxbibnames=6, giveninits=true]{biblatex}
\bibliography{main}

\newtheorem{theorem}{Theorem}[section]
\newtheorem{lemma}[theorem]{Lemma}
\newtheorem{corollary}[theorem]{Corollary}

\theoremstyle{definition}
\newtheorem{definition}[theorem]{Definition}

\theoremstyle{remark}
\newtheorem{remark}[theorem]{Remark}

\numberwithin{equation}{section}

\begin{document}

\title{Elastic Curves with Variable Bending Stiffness}

\author{Oliver Gross}
\address{University of California, San Diego, 9500 Gilman Dr, La Jolla, CA 92093, USA}
\curraddr{}
\email{ogross@ucsd.edu}
\thanks{}

\author{Ulrich Pinkall}
\address{Technische Universit\"at Berlin, Str. des 17. Juni 136, 10623, Berlin, Germany}
\curraddr{}
\email{pinkall@math.tu-berlin.de}
\thanks{}

\author{Moritz Wahl}
\address{Universit\"at Regensburg, Universit\"atsstr. 31, 93053, Regensburg, Germany}
\curraddr{}
\email{moritz.wahl@stud.uni-regensburg.de}
\thanks{}



\date{\today} 



\begin{abstract}
We study stationary points of the bending energy of curves \(\gamma\colon[a,b]\to\RR^n\) 
subject to constraints on the arc-length and the curve's holonomy while simultaneously allowing for a variable bending stiffness along the arc-length of the curve. Physically, this can be understood as a model for an elastic wire with isotropic cross-section of varying thickness. We derive the corresponding Euler-Lagrange equations for variations that are compactly supported away from the end points thus obtaining characterizations for elastic curves with variable bending stiffness. Moreover, we provide a collection of alternative characterizations, \eg, in terms of the curvature function. Adding to numerous known results relating elastic curves to dynamics, we explore connections between elastic curves with variable bending stiffness, variable length pendulums and the flow of vortex filaments with finite thickness.


\end{abstract}

\maketitle

\section{Introduction}\label{sec:intro}
In mathematics and the natural sciences alike, one-dimensional flexible structures have diverse applications ranging from architectural beams to molecular modeling. Despite a longstanding history of research, they remain a prominent research topic. 

Stationary points of the so-called \emph{bending energy} on the space of curves \(\gamma\colon[a,b]\to\RR^n\), are used to model, \eg, the bent shapes of an ideal infinitesimally thin elastic rod (without stretching)~\cite{LevienPhD}. This classical problem dates back to the 13th century and the current state of knowledge is still largely based on the work of mathematical pioneers, such as Bernoulli and Euler~\cite{LevienPhD}. Only the rigorous definition of the curvature \(\kappa\) of a planar curve \(\gamma\colon[a,b]\to\RR^2\) paved the way for their seminal work, culminating in the characterization of \emph{plane elastic curves} as the stationary points of the energy functional (\cite{CurvatureHistory}, \figref{fig:EulerElastica}) \[\tfrac{1}{2}\int\kappa^2\,ds.\]
It was soon discovered that planar elastic curves relate to a variety of other phenomena in the natural sciences~\cite{LevienPhD}. 
For regular space curves \(\gamma\colon[a,b]\to\RR^3\) the bending energy is given by 
 \[\cB(\gamma)=\tfrac{1}{2}\int_a^b\langle \tfrac{d^2\gamma}{ds^2},\tfrac{d^2\gamma}{ds^2} \rangle\,ds.\] 
 Kirchhoff realized the importance of torsion and related the problem of \emph{elastic curves} in three-dimensions with the dynamics of a spinning top~\cite{kirchhoff1859ueber}. For more in depths reviews of the history of elastic curves see, \eg,~\cite{LevienPhD, Goss:2003:Phd}.
 \begin{figure}[h]
     \centering
     \includegraphics[width = \textwidth]{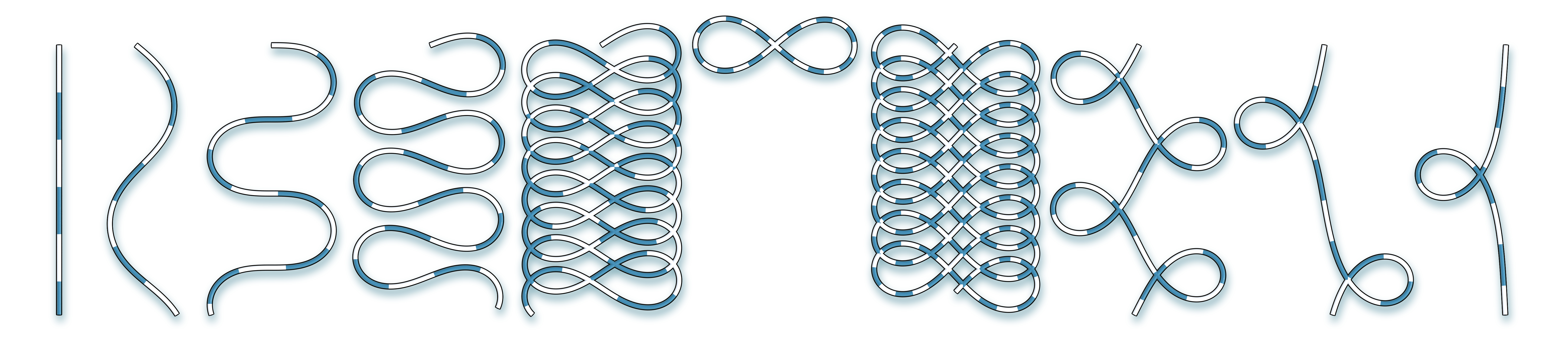}
     \caption{
     Examples of planar elastic curves, \ie, solutions to the Euler-Lagrange equations resulting from considering variations of the bending energy with a length-constraint.
     }
     \label{fig:EulerElastica}
 \end{figure}

 Elastic curves have broad utility, praised for their aesthetics and practicality. They are used as models for plant stems~\cite{ClimbingPlants2006, ClimbingPlants1998}, DNA strands~\cite{DNAmechanics2005, DNA-Rings, DNAtwist2004}, for preventing kinks and loops in marine cables~\cite{LoopFormation, CableHockling}, guiding track layouts, and even in medical applications such as surgical wires~\cite{Spillmann2010InextensibleElRods}. In design and computer graphics they serve as decorative elements in~\cite{OrnamentalCurves} and play a crucial role in simulation software, particularly in complex tasks like hair simulation~\cite{DualStripWeaving, CoRdESpillmann2007, SuperhelicesForHair, HairSimKmoch2009}. The discretization~\cite{DiscreteElasticRods, BobenkoSurisDiscreteKirchhoff} and efforts to approximate elastic curves with computationally more efficient splines~\cite{SplineApprox} form a crucial facet of current research.

Elastic curves' energy-minimizing characteristics and their direct link to material bending make them appealing for architecture and modern fabrication. Fabri\-cation-aware design and cost-effective manufacturing processes, such as active bending~\cite{ActiveBending2013}, leverage certain material properties. Bending and twisting, once challenges, are now utilized as tools, reflecting a shift ``from failure to function''~\cite{BaekReis2019}. This shift has amplified the interest in inverse problems, which seek to control material parameters to achieve specific shapes, a significant facet of elasticity research~\cite{IsotropicKirchhoffRods2018, Hafner21, Hafner23}. Economic considerations, including ease of manufacturing, transportation, and installation for curved shapes, are pivotal in architecture and fabrication. Techniques like active bending offer numerous advantages~\cite{ActiveBending2013}, and morphing structures leverage bending to attain or modify their desired shapes~\cite{TaperedElasticae2020}.

Various surface design methods extend elastic curve theory to approximate \(2\)-dimen\-sional curved shapes using networks of single rods. These approaches are instrumental in modern fabrication, including \emph{rod meshes} and \emph{gridshells}~\cite{BaekReis2019, CosseratNets, PerezFlexibleRodMeshes2015}, and \emph{deployable structures}~\cite{PanettaXShells2019, PillweinDeplGrids21}. Another modeling technique combines elastic curves with minimal surfaces, as seen in Plateau surfaces~\cite{Giusteri2017, CompDesignPlateau}. In this context, the energy of the whole system depends on the energy of the bounding elastic curve~\cite{MinimalSurfacesBernatzki, MinSurfElasticaBounded2012}. Also closed curves and ``elastic knots'' offer a range of interesting results, both in theoretical studies~\cite{Tjaden:1991:EeK, LangerSinger2006Knots} and practical applications~\cite{ElasticKnots23}.

On a theoretical level, elastic curves are known to relate to the motion of a spinning top's symmetry axis~\cite{kirchhoff1859ueber} and to satisfy the pendulum equation~\cite{PinkallGross23:DG}. Langer and Singer made significant contributions~\cite{langer1984total, LangerSinger1996}, elucidating connections between Kirchhoff rods and soliton theory, which lead to integrable Hamiltonian systems---a development built upon Hasimoto's analogy between vortex filament flow and the nonlinear Schrödinger equation~\cite{hasimoto_1972}. Specifically, the curves' evolution under the \emph{vortex-filament flow} is (up to reprarametrization) a rigid motion. Moreover, they can be described as the orbits of charged particles moving in a magnetic field~\cite{chern_knoppel_pedit_pinkall_2020, PinkallGross23:DG}. The tangent vectors of stationary points of the bending energy with no constraints on the holonomy of the curve are known to relate to the motion of the axis of a spinning top~\cite{kirchhoff1859ueber} and solve the pendulum equation~\cite{PinkallGross23:DG}.

\begin{figure}[ht]
    \centering
    \includegraphics[width=0.35\textwidth]{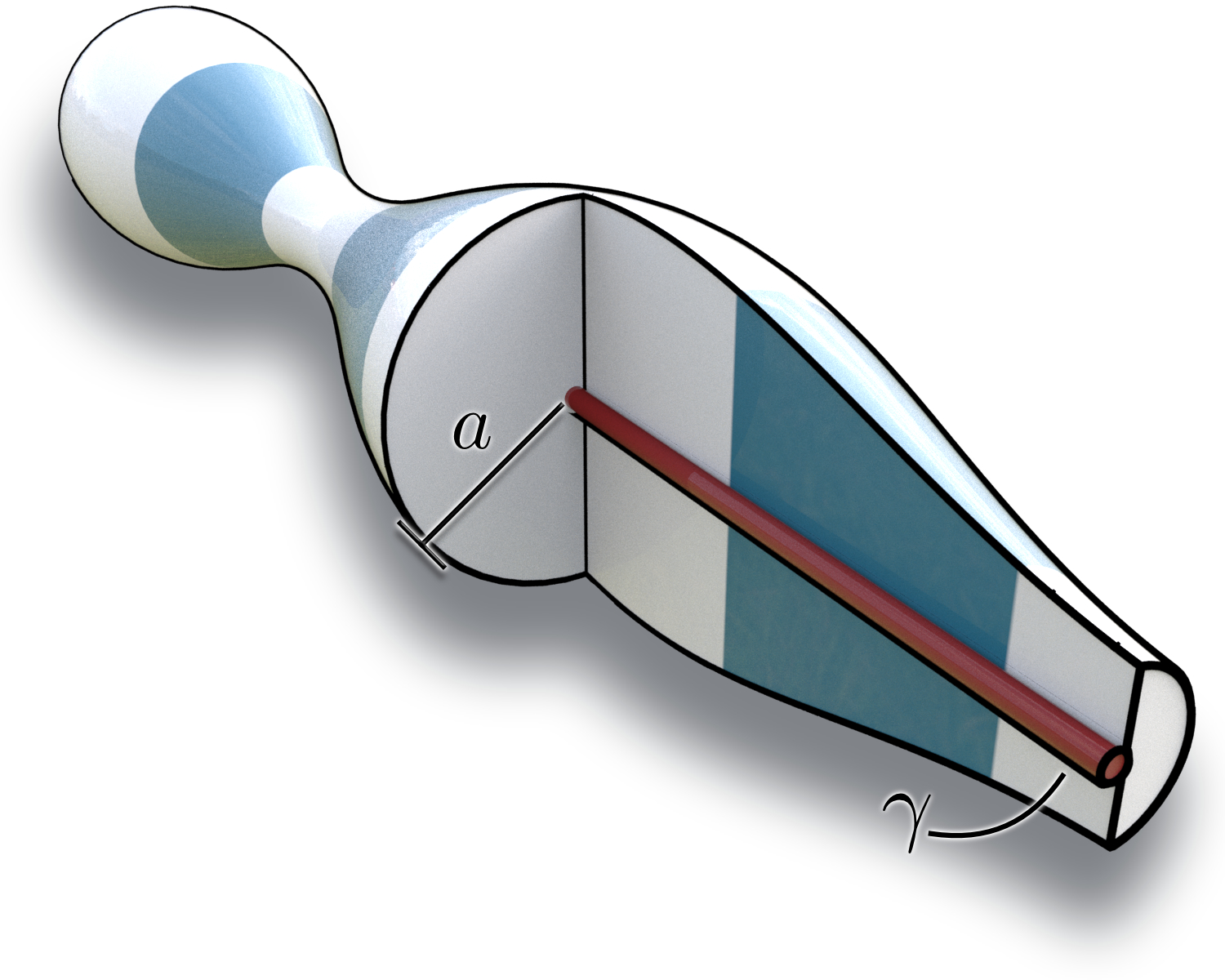}
    \hspace{2.em}
    \includegraphics[width=.45\columnwidth]{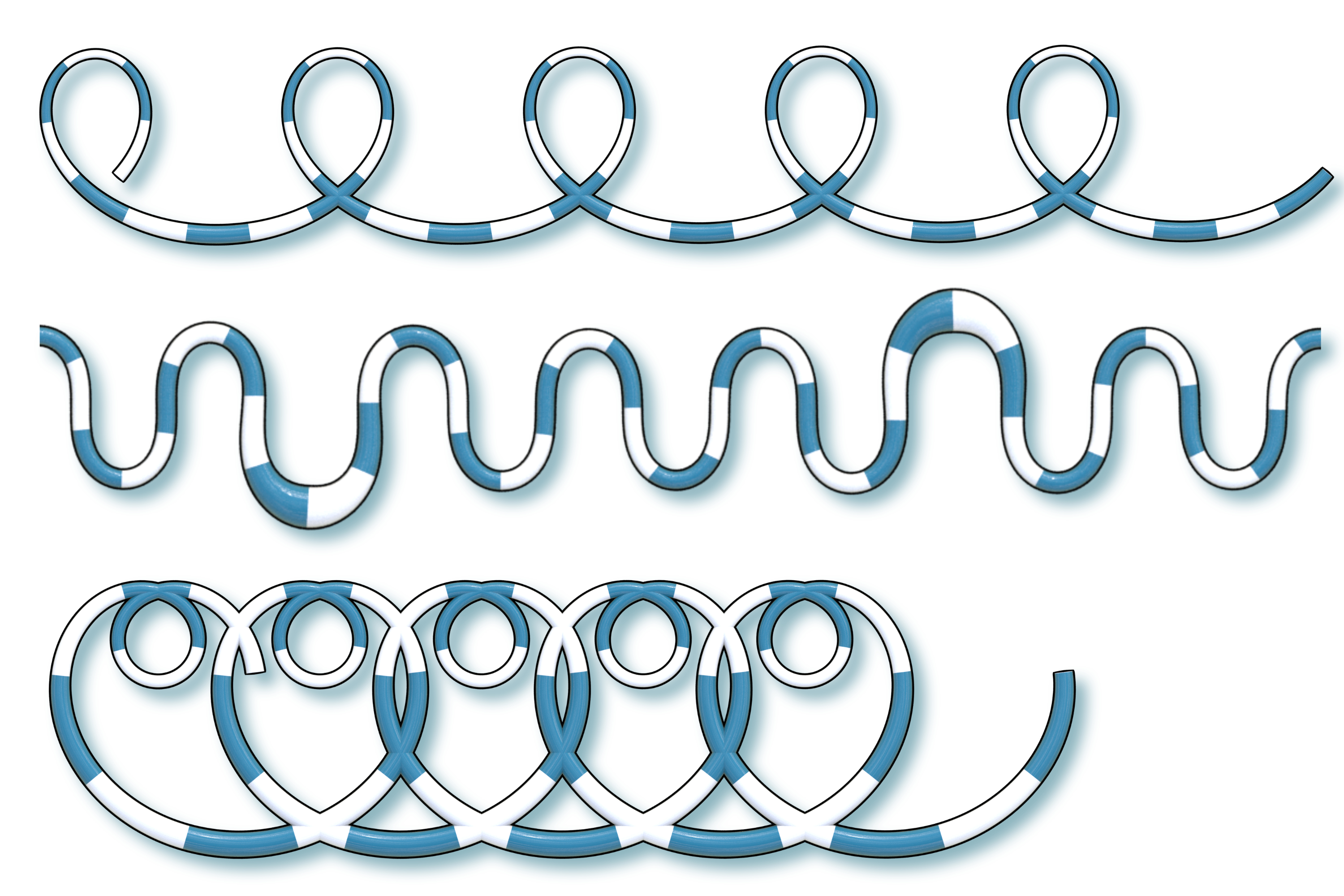}
  \caption{Elastic rods with circular cross-section of varying thickness are modeled by weighting the integrand of the bending energy with a positive function \(\varrho\) proportional to the thickness \(a>0\) (left). The resulting elastica are less bent in regions with greater stiffness (right).}
  \label{fig:VariableThicknessInset}
\end{figure}
Complementing previous work, our main focus is on space curves \(\gamma\colon[a,b]\to\RR^n\) that exhibit an isotropic resistance to bending which may vary along the arc-length of the curve. We will refer to such a curve as a curve with \emph{variable bending stiffness}. From a physical viewpoint, these elastic curves model idealized elastic rods with circular cross-section of varying diameter. The study of curves with additional free parameters such as thickness has become an important task in the natural sciences, as they appear as key structural motifs in a variety of contexts (see, \eg,~\cite{moffatt1990energy, GonzalesMaddocks:1999:GCR, cantarella2002minimum, MOULTON2013398} and references therein). The variable bending stiffness is formalized by weighting the integrand of the bending energy by a strictly positive function \[\varrho\colon[a,b]\to\RR_{>0},\] which represents the \emph{bending stiffness} of the curve \(\gamma\) along the arc-length. Consequently, the bending energy takes the form
\begin{align}
    \label{eq:BendingEnergyVariableStiffnessNotDef}
    \cB_\varrho(\gamma) = \tfrac{1}{2}\int_a^b\varrho\langle \tfrac{d^2\gamma}{ds^2},\tfrac{d^2\gamma}{ds^2} \rangle\,ds.
\end{align}
Previous investigations on this topic focused on closed planar elastic curves whose stiffness depends on an additional density variable~\cite{helmers2011snapping, brazda2021bifurcation}, or developed computational methods and applications~\cite{palmer2020anisotropic, Hafner21, Hafner23}.  

Inspired by~\cite{Hafner23}, who considered a significantly larger configuration space allowing for anisotropic cross sections and focused on the inverse problem, on a conceptual level, our approach differs significantly from previous work (see also, \eg,~\cite{Dall:2024:CCC}). We consider a somewhat more restrictive variational problem asking for isotropic cross-sections, thus remaining closer to the classical problem. However, similar to the \(2\)-dimensional case, where the step from a ``bending beam'' to solving the formal Euler-Lagrange equations reveals considerably more interesting curves, we obtain a large number of novel elastic curves with variable bending stiffness. Most importantly, we observe connections to dynamical systems, which are already well established in the classical setup~\cite{PinkallGross23:DG}. Further investigation of these aspects seems to be an exciting avenue for future research on their own.

\subsection{Structure of the Article}
The article is structured as follows: in \secref{sec:intro} we outline related work and place our own work in the context of existing literature. In \secref{sec:Preliminaries} we fix the notation and introduce preliminary definitions and theorems for further discussion. The Euler-Lagrange equations for different types of elastic curves with variable bending stiffness characterized by constraints on the arc-length and/or holonomy are derived in \secref{sec:ELEqs}. The Euler-Lagrange equations can be reformulated into equivalent characterizations, which we do in \secref{sec:EquivalentCharecterizations}, while in \secref{sec:TheCurvatureOfElasticCurves} we express them in terms of the curvature \(\kappa\colon[a,b]\to\RR^{n-1}\) function. Last, in  \secref{sec:ConnectionsWithDynamicalSystems}, we explore how established connections to dynamical systems, such as physical pendulums or the vortex filament flow, could translate to our novel setup.

\section{Preliminaries}\label{sec:Preliminaries}
In this section, we briefly introduce preliminaries and fix the notation used throughout the article. We will denote the set of all smooth functions from an interval \([a,b]\) into \(\RR^n\) by \(C^{\infty}([a,b];\RR^n)\). For our purposes we always assume that \(n\geq 2\). A \emph{curve} is a map \(\gamma\in C^{\infty}([a,b];\RR^n)\) with \(\gamma' \neq 0\).  Denoting the standard Euclidean norm on \(\RR^n\) by \(|\cdot|\), we refer to \(|\gamma'|\) as the \emph{arc-length} of \(\gamma\)  and the \emph{length} of a curve is given by
    \begin{align*}
        \cL(\gamma)\coloneqq \int_a^b|\gamma'|.
    \end{align*} 
A curve is said to be parameterized by \emph{arc-length} if \(|\gamma'|\equiv1\). The \emph{derivative} of a function \(g\in C^{\infty}([a,b];\RR^k)\) \emph{with respect to the arc-length} of \(\gamma\in C^{\infty}([a,b];\RR^n)\) is defined as \(\tfrac{dg}{ds}\coloneqq \tfrac{g'}{|\gamma'|}\).
In particular, the \emph{unit tangent} vector field \(T\in C^{\infty}([a,b];S^{n-1})\) of a curve is given by \(T\coloneqq \tfrac{d\gamma}{ds}\).

\subsection{The Bending Energy with Bending Stiffness}
The norm of the derivative with respect to the arc-length of the unit tangent vector, measures the failure of the curve  \(\gamma\) to be a straight line segment. Integrating the squared norm of this error over the arc-length therefore measures how much the curve bends in space. Consequently the \emph{bending energy} of a curve \(\gamma\in C^{\infty}([a,b];\RR^n)\) is given by 
    \begin{align}
         \label{eq:BendingEnergy}
        \cB(\gamma)\coloneqq\tfrac{1}{2}\int_a^b\langle \tfrac{dT}{ds},\tfrac{dT}{ds} \rangle\,ds.
    \end{align}
Historically, the motivation for studying elastic curves has been to gain a better understanding of the shapes that a thin elastic cable takes when the end points are held fixed. We extend the theory in the sense that we do not assume for a cable of constant thickness, but allow for varying thickness which is accounted for by weighting the bending energy density along the curve.
\begin{definition}[Bending energy with bending stiffness]\label{def:BendingEnergyVariableStiffness}
    The \emph{bending energy} of a curve \(\gamma\colon[a,b]\to\RR^n\) with \emph{bending stiffness} \(\varrho\colon[a,b]\to\RR_{>0}\) is given by 
    \begin{align}
         \label{eq:BendingEnergyVariableStiffness}
        \cB_\varrho(\gamma)\coloneqq\tfrac{1}{2}\int_a^b\varrho\langle \tfrac{dT}{ds},\tfrac{dT}{ds} \rangle\,ds.
    \end{align}
\end{definition}

\subsection{Holonomy of a Space Curve}
A normal field \(Z\colon[a,b]\to\RR^n\) satisfies \(\langle Z, T\rangle = 0\) and is said to be \emph{parallel} if \(Z'=\lambda T\) for some \(\lambda\in C^\infty(M)\). By solving a corresponding initial value problem any normal vector to a curve can be extended to a unique parallel normal field along \(\gamma\)~\cite[Thm. 4.3]{PinkallGross23:DG}. The vector \(Z_b\) of the solution \(Z\) to this initial value problem for an initial normal vector \(Z_0\) at \(\gamma(a)\) is called the \emph{parallel transport} of \(Z_0\) along the curve \(\gamma\).

\begin{definition}\label{def:TotalTorsion}
    For a curve \(\gamma\in C^{\infty}([a,b];\RR^3)\) let \(W=(W_a, W_b)\in T(a)^\perp\times T(b)^\perp\) be a pair of unit vectors. Then, the \emph{holonomy} of the curve \(\gamma\) with respect to \(W\) is the unique angle \(\mathcal{T}_W\in {\RR}/_{2\pi\ZZ}\) such that 
    \begin{align}
        Z_b=\cos(\mathcal{T}_W)\,W_b + \sin(\mathcal{T}_W)\, T(b)\times W_b,
    \end{align}
    where \(Z_b\) is the unique vector obtained from a parallel transport of \(W_a\) along \(\gamma\).
\end{definition}

Note that a curve's holonomy does not depend on the bending stiffness \(\varrho\), but only the geometry of the curve \(\gamma\).

\subsection{Variations of Curves, Length and Holonomy}
In the forthcoming sections we will define a hierarchy of ``elastic curves'' as stationary points of bending energy with bending stiffness \eqref{eq:BendingEnergyVariableStiffness} under perturbations with constrained length, arc-length or holonomy. For \(g\in C^{\infty}([a,b];\RR^n)\) and \(\epsilon >0\), a \emph{smooth variation} of \(g\) is a one-parameter family 
        \begin{align}
            t\mapsto g_t \in C^{\infty}([a,b];\RR^n),  
        \end{align}
        where \(t\in [-\epsilon,\epsilon]\) and which satisfies \(g_0=g\) and such that the map
        \begin{align*}
            [-\epsilon,\epsilon]\times[a,b]\to\RR^n,\ (t,x)\mapsto \gamma_t(x)
        \end{align*}
        is smooth. Given a smooth variation of a map \(g\in C^{\infty}([a,b];\RR^n)\), also \(t\mapsto g'_{t}\), as well as \(t\mapsto \mathring g_{t}\), where \(\mathring g_{t}\in C^{\infty}([a,b];\RR^n)\) is defined as \[\mathring g_t(x)\coloneqq\tfrac{d}{d\tau}\big\vert_{\tau=t}g_\tau(x),\]
are smooth. To simplify the notation, we will omit the index when we evaluate at time \(t=0\) and write \(\mathring g = \mathring g_0\). Moreover, we denote the variation of a smooth functional \(\cF\) on \(C^{\infty}([a,b];\RR^n)\) corresponding to a smooth variation \(t\mapsto g_t\) with variational vector field \(\mathring g\) by \[d\cF(\mathring g)\coloneqq \tfrac{d}{d\tau}\big\vert_{\tau=0}\cF(g_\tau).\]
Our main object of interest are smooth variations \(t\mapsto\gamma_t\in C^{\infty}([a,b];\RR^n)\) of curves, for which we refer to 
\begin{align}
    \mathring \gamma\in C^{\infty}([a,b];\RR^n)
\end{align}
as the \emph{variational vector field}. A variation \(t\mapsto \gamma_t\) of a curve \(\gamma\in C^{\infty}([a,b];\RR^n)\) is said to have \emph{compact support in the interior} of \([a,b]\) if there is \(\delta>0\) such that for all \(x \in [a, a + \delta] \cup [b - \delta, b]\) it holds that \(\gamma_t (x) = \gamma (x)\). We denote the vector space of functions with support in the interior of \([a,b]\) by \(C_0^{\infty}([a,b];\RR^n)\) and treat variations and corresponding variational vector fields as synonyms. Some useful identities are collected in 
\begin{lemma}[{\cite[Ch. 2]{PinkallGross23:DG}}]\label{thm:HelpfulDerivativeRules}
    Let \(\gamma\in C^{\infty}([a,b];\RR^n)\) and \(g\in C^{\infty}([a,b];\RR^n)\). Then,
    \begin{enumerate}
        \item \((\mathring g)' = (g')^{\mathring{}}\).
        \item \((ds)^{\mathring{}}= \langle \tfrac{d \mathring\gamma}{ds}, T \rangle ds\).
        \item \(\left(\tfrac{dg}{ds}\right)^{\mathring{}} = \tfrac{d\mathring g}{ds} - \langle \tfrac{d \mathring\gamma}{ds}, T \rangle\tfrac{dg}{ds}\). \label{eq:Lem2-3partiii}
    \end{enumerate}
\end{lemma}

\begin{theorem}
    Let \(\gamma\in C^{\infty}([a,b];\RR^n)\) and  \(\mathring\gamma\in C^{\infty}([a,b];\RR^n)\). Then,  
    \begin{align*}
        d\cL(\mathring\gamma) = \langle \mathring\gamma,T\rangle\big\vert_a^b - \int_a^b\langle \mathring\gamma, \tfrac{dT}{ds} \rangle\, ds.
    \end{align*}
\end{theorem}

\begin{proof}
    By \lemref{thm:HelpfulDerivativeRules} and the fundamental theorem of calculus we have
    \begin{align*}
        d\cL(\mathring\gamma)=\int_a^b \langle \tfrac{d \mathring\gamma}{ds}, T \rangle\, ds = \int_a^b\left( \tfrac{d}{ds}\langle \mathring\gamma,T\rangle - \langle \mathring\gamma, \tfrac{dT}{ds}\rangle\right)ds = \langle \mathring\gamma,T\rangle\big\vert_a^b - \int_a^b\langle \mathring\gamma, \tfrac{dT}{ds} \rangle\, ds .
    \end{align*}
\end{proof}
We can also compute the variational gradient of the holonomy. 

\begin{theorem}\label{thm:VariationTotalTorsion}
    Let \(\gamma\in C^{\infty}([a,b];\RR^n)\) and \(\mathring\gamma\in C^{\infty}_0([a,b];\RR^n)\). Then, independent of the choice of \(W=(W_a,W_b)\) it holds that
    \begin{align}
        \label{eq:VariationTotalTorsion}
        d\cT_{W}(\mathring\gamma)=
        \int_a^b\langle\mathring\gamma, T\times(\tfrac{dT}{ds})'\rangle.
    \end{align}
\end{theorem}

\begin{proof}
    This follows from restricting \cite[Thm. 5.3]{PinkallGross23:DG} to variations in \(C^{\infty}_0([a,b];\RR^n)\).
\end{proof}

\section{Euler-Lagrange Equations}\label{sec:ELEqs}
With all necessary preliminaries in place, we start this section with computing the variational formula of the bending energy with bending stiffness \eqref{eq:BendingEnergyVariableStiffness} for a curve \(\gamma\in C^{\infty}([a,b];\RR^n)\) with bending stiffness  \(\varrho\in C^{\infty}([a,b];\RR_{>0})\) under a general variation \(\mathring\gamma\in C^{\infty}([a,b];\RR^n)\).
\begin{theorem}
[Variational formula for the bending energy]\label{thm:VaritionalGradientBendingEnergy}
Let \(\gamma\in C^{\infty}([a,b];\RR^n)\) with bending stiffness  \(\varrho\in C^{\infty}([a,b];\RR_{>0})\). Then, the variation of the bending energy with bending stiffness \eqref{eq:BendingEnergyVariableStiffness} with respect to \(\mathring\gamma\in C^{\infty}([a,b];\RR^n)\) is given by
\begin{align}
    \label{eq:VariationalGradientBendingEnergyFull}
    d\cB_\varrho(\mathring\gamma)=
    \big[ 
        \varrho \langle \tfrac{d\mathring\gamma}{ds} , \tfrac{dT}{ds} \rangle -
        \langle\mathring\gamma , 
            \varrho \tfrac{d^2T}{ds^2} +
            \tfrac{3}{2} \varrho \langle \tfrac{dT}{ds} , \tfrac{dT}{ds} \rangle T
            + \tfrac{d\varrho}{ds} \tfrac{dT}{ds}
        \rangle
    \big]_a^b +\int_a^b\langle\mathring\gamma, \cG^\cB\rangle\,ds,
    \end{align}
    where
    \begin{align}\label{eq:FreeElasticGeneralForT}
       \cG^\cB \coloneqq      \varrho \tfrac{d^3T}{ds^3} + 3 \varrho \big\langle \tfrac{d^2T}{ds^2} , \tfrac{dT}{ds} \big\rangle T + \tfrac{3}{2} \varrho \big\langle \tfrac{dT}{ds} , \tfrac{dT}{ds} \big\rangle \tfrac{dT}{ds}
            +
            \tfrac{d^2\varrho}{ds^2} \tfrac{dT}{ds}
            + \tfrac{d\varrho}{ds}
                ( 
                2 \tfrac{d^2T}{ds^2}
                + \tfrac{3}{2} \big\langle \tfrac{dT}{ds} , \tfrac{dT}{ds} \big\rangle T 
                ).
\end{align}
\end{theorem}
\begin{proof}
First note that \(\langle\tfrac{dT}{ds},T\rangle=0\) and \lemref{thm:HelpfulDerivativeRules} \ref{eq:Lem2-3partiii} applied to \(g=\gamma\) gives \(\mathring T = \tfrac{d\mathring\gamma}{ds} - \langle \tfrac{d\mathring\gamma}{ds}, T\rangle T\). Therefore, the variation of Bending Energy with bending stiffness is 
\begin{align*}
    d\cB_\varrho(\mathring\gamma)
        &= \tfrac{1}{2}\int_a^b \left(\varrho\langle \tfrac{dT}{ds},\tfrac{dT}{ds}\rangle\,ds\right)^\circ \\
        &= \int_a^b \varrho ( \langle (\tfrac{dT}{ds})^\circ,\tfrac{dT}{ds}\rangle\, ds + \tfrac{1}{2}\langle \tfrac{dT}{ds},\tfrac{dT}{ds}\rangle\,d\mathring{s})\\
        &= \int_a^b \varrho ( \langle \tfrac{d\mathring{T}}{ds} - \langle \tfrac{d\mathring{\gamma}}{ds}, T\rangle \tfrac{dT}{ds},\tfrac{dT}{ds}\rangle\, ds + \tfrac{1}{2}\langle \tfrac{d\mathring\gamma}{ds},T\rangle \langle \tfrac{dT}{ds},\tfrac{dT}{ds}\rangle \,ds)\\
        &= \int_a^b \varrho ( \langle \tfrac{d\mathring{T}}{ds},\tfrac{dT}{ds}\rangle - \tfrac{1}{2}\langle \tfrac{d\mathring\gamma}{ds},T\rangle \langle \tfrac{dT}{ds},\tfrac{dT}{ds}\rangle ) \,ds\\
        &= \int_a^b \varrho ( \langle \tfrac{d}{ds}( \tfrac{d\mathring\gamma}{ds}- \langle \tfrac{d\mathring\gamma}{ds}, T\rangle T ),\tfrac{dT}{ds}\rangle - \tfrac{1}{2} \langle \tfrac{d\mathring\gamma}{ds},T\rangle \langle \tfrac{dT}{ds},\tfrac{dT}{ds}\rangle) \,ds\\
        &= \int_a^b \varrho ( \langle \tfrac{d^2\mathring\gamma}{ds^2},\tfrac{dT}{ds}\rangle - \tfrac{3}{2}\langle \tfrac{d\mathring\gamma}{ds},T\rangle \langle \tfrac{dT}{ds},\tfrac{dT}{ds}\rangle) \,ds\\
\end{align*}
We split the integral just to keep computations more clear. Using integration by parts we obtain
\begin{align*}
    \int_a^b \varrho
    \langle \tfrac{d^2 \mathring\gamma}{ds^2} , \tfrac{dT}{ds} \rangle
    ds &= 
    \int_a^b \tfrac{d}{ds}
    (
        \varrho
        \langle \tfrac{d\mathring\gamma}{ds} , \tfrac{dT}{ds} \rangle
    )
    - \tfrac{d\varrho}{ds}
    \langle \tfrac{d\mathring\gamma}{ds} , \tfrac{dT}{ds} \rangle
    - \varrho
    \langle \tfrac{d\mathring\gamma}{ds} , \tfrac{d^2T}{ds^2} \rangle
    \,ds
    \\
    &=
    \big[ \varrho \langle \tfrac{d\mathring\gamma}{ds} , \tfrac{dT}{ds} \rangle\big]_a^b
    + \int_a^b 
    - \tfrac{d}{ds}
    (
        \tfrac{d\varrho}{ds}
        \langle \mathring\gamma , \tfrac{dT}{ds} \rangle
    )
    + \tfrac{d^2\varrho}{ds^2} \langle \mathring\gamma , \tfrac{dT}{ds} \rangle
    + \tfrac{d\varrho}{ds} \langle \mathring\gamma , \tfrac{d^2T}{ds^2} \rangle
    \\
    &\quad 
    - \tfrac{d}{ds}
    (
        \varrho \langle \mathring\gamma , \tfrac{d^2T}{ds^2} \rangle
    )
    + \tfrac{d\varrho}{ds} \langle \mathring\gamma , \tfrac{d^2T}{ds^2} \rangle
    + \varrho \langle \mathring\gamma , \tfrac{d^3T}{ds^3} \rangle 
    \,ds\\
    &=
    \big[
        \varrho \langle \tfrac{d\mathring\gamma}{ds} , \tfrac{dT}{ds} \rangle\
        -
        \tfrac{d\varrho}{ds} \langle \mathring\gamma , \tfrac{dT}{ds} \rangle
        -
        \varrho \langle \mathring\gamma , \tfrac{d^2T}{ds^2} \rangle
    \big]_a^b
    \\
    &\quad
    + \int_a^b
    \tfrac{d^2\varrho}{ds^2} \langle \mathring\gamma , \tfrac{dT}{ds} \rangle
    + \tfrac{d\varrho}{ds} \langle \mathring\gamma , \tfrac{d^2T}{ds^2} \rangle
    + \tfrac{d\varrho}{ds} \langle \mathring\gamma , \tfrac{d^2T}{ds^2} \rangle
    + \varrho \langle \mathring\gamma , \tfrac{d^3T}{ds^3} \rangle
    \,ds
    \\
    &=
    \big[
        \varrho \langle \tfrac{d\mathring\gamma}{ds} , \tfrac{dT}{ds} \rangle\
        -
        \langle
            \mathring\gamma , \varrho \tfrac{d^2T}{ds^2} + \tfrac{d\varrho}{ds} \tfrac{dT}{ds}
        \rangle
    \big]_a^b
    \\
    &\quad
    + \int_a^b
    \langle
        \mathring\gamma ,
        \varrho \tfrac{d^3T}{ds^3}
        + 2 \tfrac{d\varrho}{ds} \tfrac{d^2T}{ds^2}
        + \tfrac{d^2\varrho}{ds^2} \tfrac{dT}{ds}
    \rangle
    \,ds.
\end{align*}
Similarly, for the second integral we get
\begin{align*}
    \int_a^b \tfrac{3}{2} \varrho
    \langle \tfrac{d\mathring\gamma}{ds} , T \rangle
    \langle \tfrac{dT}{ds} , \tfrac{dT}{ds} \rangle
    ds &=
    \int_a^b \tfrac{d}{ds}
    ( 
        \tfrac{3}{2} \varrho\langle\mathring\gamma , T \rangle
        \langle \tfrac{dT}{ds} , \tfrac{dT}{ds} \rangle
    )
    \\
    &\quad - \tfrac{3}{2} \varrho \langle \mathring\gamma , \tfrac{dT}{ds} \rangle
    \langle \tfrac{dT}{ds} , \tfrac{dT}{ds} \rangle
    - 3 \varrho \langle \mathring\gamma , T \rangle
    \langle \tfrac{d^2T}{ds^2} , \tfrac{dT}{ds} \rangle
    \\
    &\quad - \tfrac{3}{2} \tfrac{d\varrho}{ds} \langle \mathring\gamma , T \rangle
    \langle \tfrac{dT}{ds} , \tfrac{dT}{ds} \rangle
    \,ds
    \\
    &=
    \big[
        \tfrac{3}{2} \varrho \langle \mathring\gamma , T \rangle \langle \tfrac{dT}{ds} , \tfrac{dT}{ds} \rangle
    \big]_a^b
    \\
    &\quad 
    - \int_a^b
    \langle
        \mathring\gamma ,
        \tfrac{3}{2} \varrho \langle\tfrac{dT}{ds},\tfrac{dT}{ds}\rangle \tfrac{dT}{ds}
        + 3 \varrho \langle \tfrac{d^2T}{ds^2} \tfrac{dT}{ds} \rangle T
        + \tfrac{3}{2} \tfrac{d\varrho}{ds} \langle\tfrac{dT}{ds},\tfrac{dT}{ds}\rangle T
    \rangle
    \,ds.
\end{align*}
Adding the two results yields the claim.
\end{proof}

\subsection{Free Elastic Curves with Bending Stiffness}
\thmref{thm:VaritionalGradientBendingEnergy} bears several significant implications. When we imagine an (initially perfectly straight) elastic wire, it naturally wants to minimizes its bending energy. Holding a piece of such wire in our hands, we fix its end-points (and in fact even its tangent directions at the end points). Therefore, from a physical point of view, it is reasonable to restrict our attention to variations of \eqref{eq:BendingEnergyVariableStiffness} with compact support in the interior of \([a,b]\).

Already in classical theory, free elastic curves\footnote{\ie, unconstrained stationary points of the bending energy} take on a special role because, up to scaling and positioning, there is only one such curve. It turns out that even in our generalized setup, our results in \secref{sec:TheCurvatureOfElasticCurves} imply that the constant stiffness ones remain the only solution for curves that are not straight line segments. 
\begin{theorem}
   Free elastic curves are either straight line segments or have constant bending stiffness.
\end{theorem}

\subsection{Elastic Curves with Bending Stiffness}
Our investigations are inspired by deformations of physical rods or cables, possibly with non-uniform thickness distribution. We assume for the thickness to be prescribed along the arc length of the cable. 
\begin{figure}[h]
   \centering
   \includegraphics[width=.9\columnwidth]{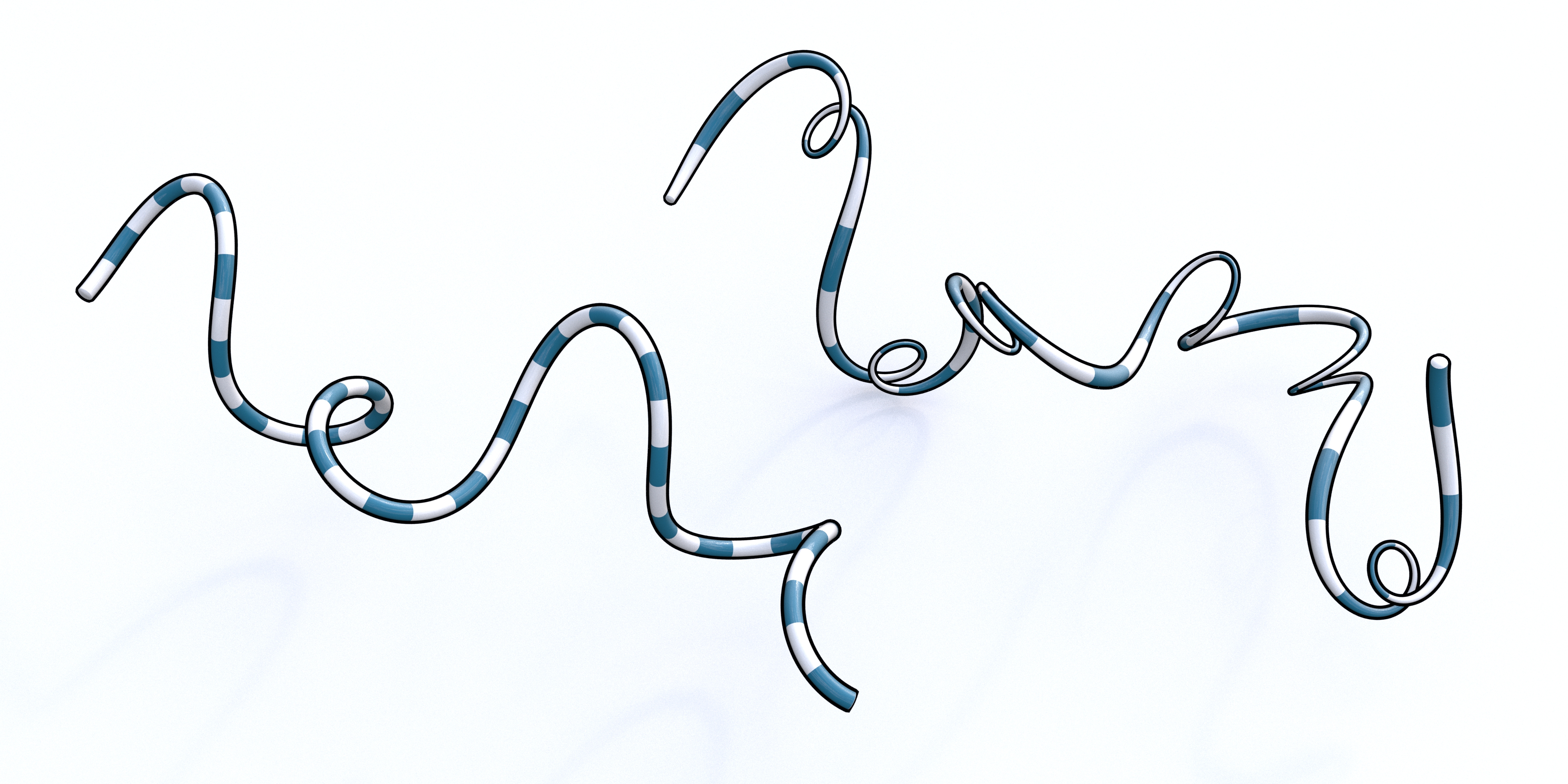}
   \caption{\label{fig:curve1_cst_and_sin}A comparison between an elastic curve with constant bending stiffness \(\varrho\) obtained from integrating \teqref{eq:PendulumEqDefiningEq} (left) and a corresponding curve obtained from the same initial conditions, but with a modulated bending stiffness \(\varrho\) (right).
   }
\end{figure}
To incorporate this assumption into our model we constrain the set of admissible variations to the arc-length parameterized curves. By standard means of calculus of variations, this can be achieved by introducing a suitable Lagrange multiplier~\cite{Singer2008Lectures, zeidler2012applied}. More specifically,
a curve \(\gamma\in C^{\infty}([a,b];\RR^n)\) with bending stiffness \(\varrho\in C^{\infty}([a,b];\RR_{>0})\) is a stationary point of \eqref{eq:BendingEnergyVariableStiffness} subject to the pointwise constraint 
\[\langle \gamma ' , \gamma ' \rangle - 1 = 0 \]
if and only if there exists a function \(\Lambda \in C^{\infty}([a,b];\RR)\) such that \(\gamma\) is a stationary point of 
\begin{align}
    \label{eq:BendingEnergyVariableStiffnessUnitSpeedConstraint}
    \cB^\Lambda_\varrho(\gamma)\coloneqq\int_a^b\tfrac{1}{2}\varrho\langle T',T' \rangle + \Lambda(\langle \gamma ' , \gamma ' \rangle - 1) .
\end{align}

\begin{remark}
Note that the arc-length constraint is in fact appropriate when dealing with non-uniform bending stiffness. Unlike as for a balloon animal, to mimic the behavior of a physical cable, the thickness shall not be redistributed along the curve. A sole constraint on the length does not rule out such scenarios, though clearly the arc-length constraint also automatically constraints the length of the curve \(\gamma\).
\end{remark}

\begin{definition}[Elastic curve with bending stiffness]
    An arc-length par\-ametrized curve \(\gamma\in C^{\infty}([a,b];\RR^n)\) with bending stiffness \(\varrho\in C^{\infty}([a,b];\RR_{>0})\) is \emph{elastic with bending stiffness} if it is a critical point of the energy in \teqref{eq:BendingEnergyVariableStiffness} under all variations \(\mathring\gamma\in C^\infty_0([a,b];\RR^n)\) constraining the arc-length of the curve.
\end{definition}

\begin{theorem} \label{thm:TorsionFreeElasticsODE}
A curve \(\gamma\in C^{\infty}([a,b];\RR^n)\) with bending stiffness \(\varrho\in C^{\infty}([a,b];\RR_{>0})\) is an elastic curve if and only if there is a \(\Lambda \in C^{\infty}([a,b];\RR)\) such that \[\cG^\cB_\Lambda \coloneqq \cG^\cB - 2 \Lambda ' T - 2 \Lambda T' = 0.\] In terms of the unit tangent field \(T\) this can be expressed by
\begin{align*}
    \varrho T'''
    + 3 \varrho \langle T' , T'' \rangle T
    + \tfrac{3}{2} \varrho \langle T' , T' \rangle T'
    + \varrho'' T'
    + 2 \varrho' T''
    + \tfrac{3}{2} \varrho' \langle T' , T' \rangle T
    - \Lambda ' T
    - \Lambda T'
    = 0, 
\end{align*}
or equivalently, in terms of \(\gamma\),
\begin{equation*}
    \varrho \gamma''''
    +3  \varrho \langle \gamma''' , \gamma'' \rangle \gamma'
    + \tfrac{3}{2} \varrho \langle \gamma'' , \gamma'' \rangle \gamma''
    + \varrho'' \gamma''
    + 2 \varrho' \gamma'''
    + \tfrac{3}{2} \varrho' \langle \gamma'' , \gamma'' \rangle \gamma'
    - \Lambda' \gamma'
    - \Lambda \gamma''
    = 0.
\end{equation*}
\end{theorem}
\begin{figure}[h]
   \centering
   \includegraphics[width=.9\columnwidth]{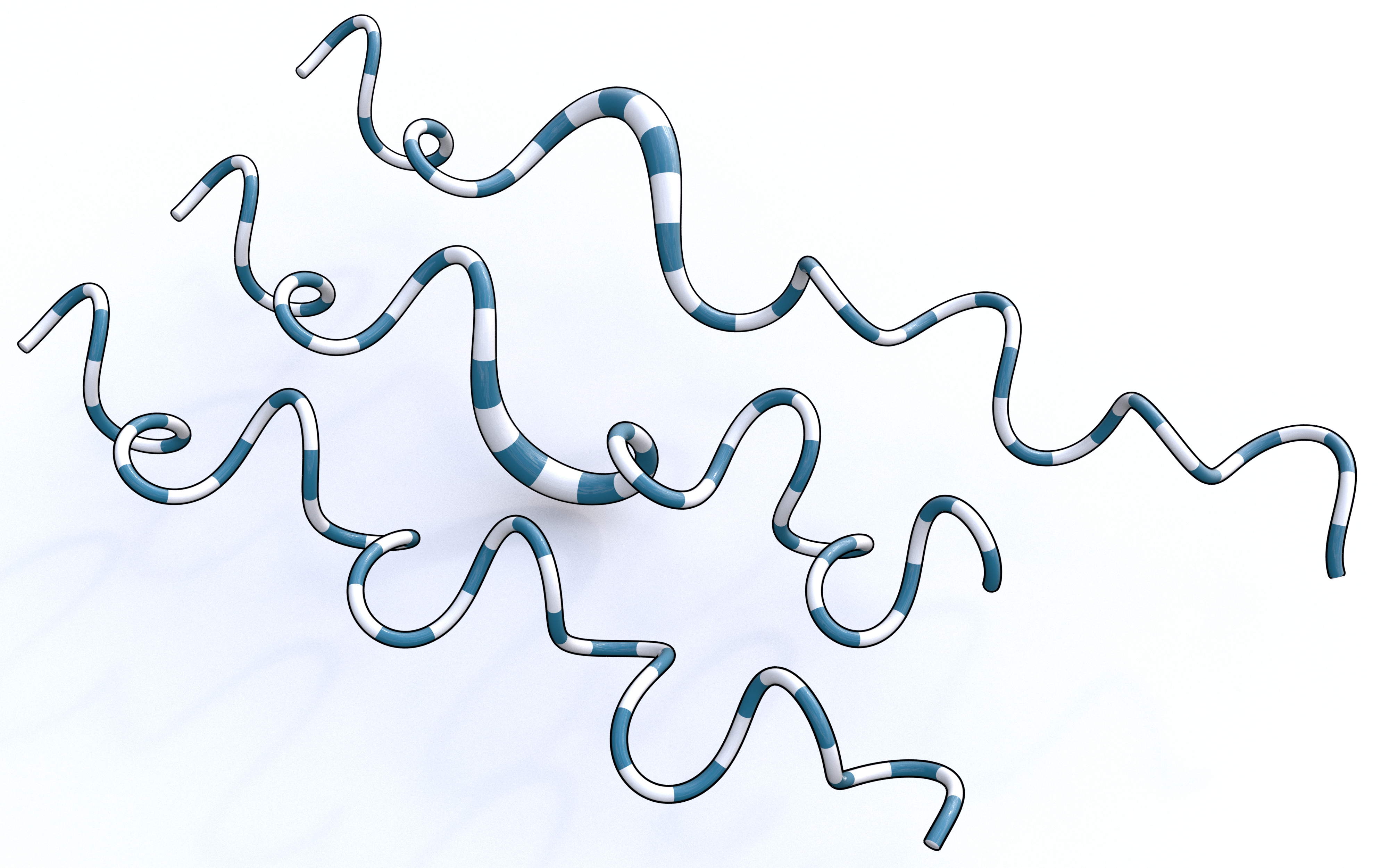}
   \caption{\label{fig:curve1_agnesi_triple.}
   A comparison between an elastic curve with constant bending stiffness \(\varrho\) obtained from integrating \teqref{eq:PendulumEqDefiningEq} (front) and corresponding curves obtained from the same initial conditions, but with a modulated bending stiffness \(\varrho\) (middle, back).
   }
\end{figure}
\begin{proof}
    Taking the time derivative of the augmented functional in \teqref{eq:BendingEnergyVariableStiffnessUnitSpeedConstraint} and using integration by parts we compute
    \begin{align*}
        0 &= \int_a^b \langle \mathring\gamma , \cG^\cB \rangle + \int_a^b 2 \Lambda \langle \mathring\gamma' , T \rangle\\
        &= \int_a^b \langle \mathring\gamma , \cG^\cB \rangle + \int_a^b \big( 2 \Lambda \langle \mathring\gamma , T \rangle \big)' - 2 \Lambda ' \langle \mathring\gamma , T \rangle - 2 \Lambda \langle \mathring\gamma , T' \rangle\\
        &= \int_a^b \langle \mathring\gamma , \cG^\cB \rangle + \int_a^b - 2 \Lambda ' \langle \mathring\gamma , T \rangle - 2 \Lambda \langle \mathring\gamma , T' \rangle\\
        &= \int_a^b \langle \mathring\gamma , \cG^\cB - 2 \Lambda ' T - 2 \Lambda T' \rangle .
    \end{align*}
    Note that the boundary terms vanish since we consider variations \(\mathring\gamma \in C^{\infty}_0([a,b];\RR^n)\) and the claim follows. 
\end{proof}

\subsection{Holonomy Constrained Elastic Curves with Bending Stiffness}
In this section we will restrict our attention to curves \(\gamma\in C^{\infty}([a,b];\RR^3)\) in \(\RR^3\). As outlined in \secref{sec:intro}, Kirchhoff established torsion as an important parameter when modeling the shapes of elastic wires. One can show that, for the purpose of calculus of variations, constraints on a curve's ``total torsion'' or holonomy are equivalent~\cite{PinkallGross23:DG} and the notion of framed curves is not necessarily required. Therefore, we will refer to stationary points of \eqref{eq:BendingEnergyVariableStiffness} under arc-length preserving variations and with constrained holonomy (\defref{def:TotalTorsion}) as \emph{holonomy constrained elastic curves}\footnote{They are also known as Kirchhoff elastica.}. 

\begin{definition}[Holonomy constrained elastic curve with bending stiffness]
    A curve \(\gamma\in C^{\infty}([a,b];\RR^3)\) with \emph{bending stiffness} \(\varrho\in C^{\infty}([a,b];\RR_{>0})\) is said to be a \emph{holonomy constrained elastic curve} if it is a critical point of the energy in \teqref{eq:BendingEnergyVariableStiffness} under all variations \(\mathring\gamma\in C^\infty_0([a,b];\RR^n)\) constraining the arc-length and holonomy of the curve.
\end{definition}

We can again derive the Euler-Lagrange equation characterizing holonomy constrained elastic curves by introducing a suitable Lagrange multiplier constraining the holonomy (\thmref{thm:VariationTotalTorsion}). For holonomic constraints such as fixed length, or holonomy, the Lagrange multipliers are in fact constants~\cite[Sec. 2]{PinkallGross23:DG}.

\begin{figure}[h]
   \centering
   \includegraphics[width=\columnwidth]{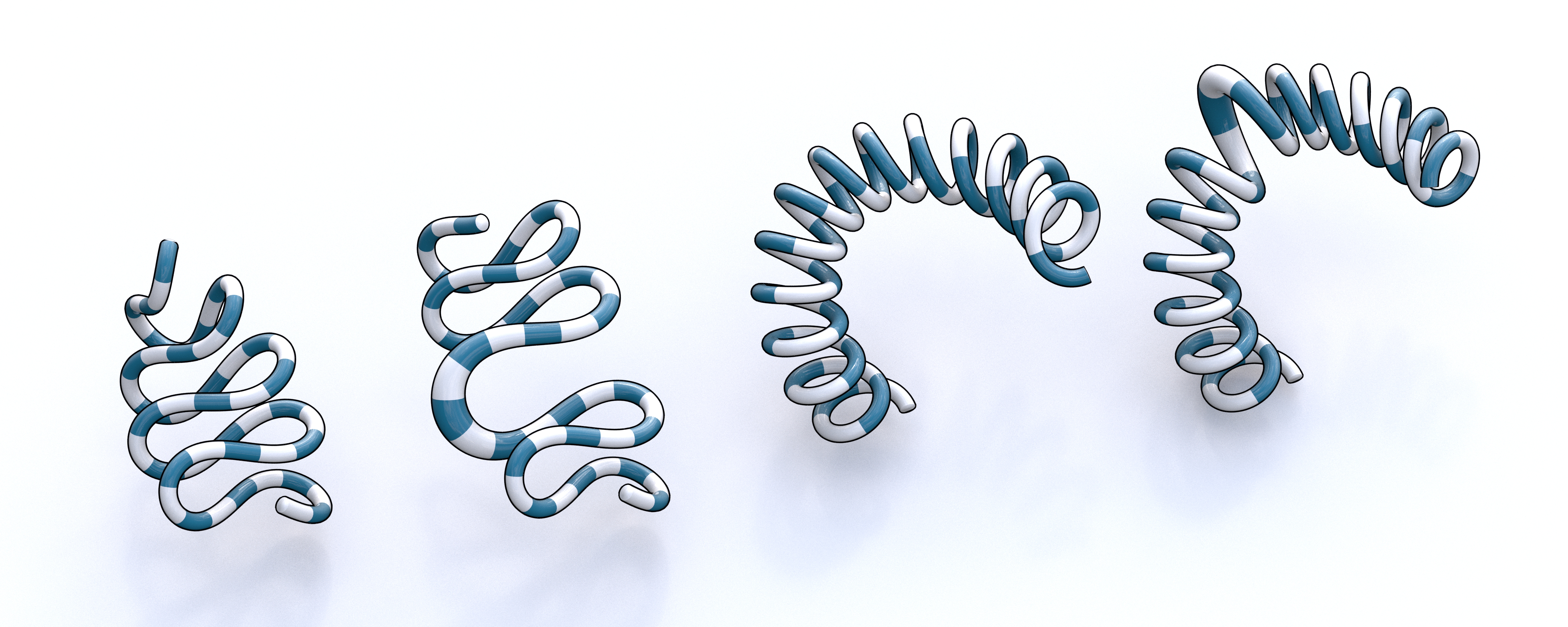}
   \caption{\label{fig:Elastic_HeavyTorsion}
   Two elastic holonomy constrained curves with constant bending stiffness obtained from integrating condition \ref{Equiv4} in \thmref{thm:3dEquivalences} with a constant bending stiffness \(\varrho\) (left, middle left) and corresponding curves obtained from the same initial conditions, but with a modulated bending stiffness \(\varrho\) (middle right, right).
   }
\end{figure}

\begin{theorem} \label{thm:ElasticCurvesODE}
A curve \(\gamma\in C^{\infty}([a,b];\RR^3)\) with bending stiffness \(\varrho\in C^{\infty}([a,b];\RR_{>0})\) is a holonomy constrained elastic curve if and only if there is a smooth function \(\Lambda\colon [a,b] \rightarrow \RR\) and a constant \(\mu \in \RR\) such that
\[\cG^\cB_{\Lambda,\mu} \coloneqq\cG^\cB_\Lambda -\mu T \times T''= \cG^\cB - 2 \Lambda ' T - 2 \Lambda T' -\mu T \times T'' = 0.\] In terms of the unit tangent field \(T\) this can be expressed by
    \begin{equation*}
        \varrho T'''-\varrho \langle T''' , T \rangle T
        +\tfrac{3}{2} \varrho \langle T' , T' \rangle T'
        +\varrho '' T'
        +\varrho ' ( 2 T''+\tfrac{3}{2} \langle T' , T' \rangle T )
        -\Lambda ' T
        -\Lambda T'
        -\mu T \times T''
        = 0,
    \end{equation*}
    or equivalently, in terms of \(\gamma\),
\begin{equation*}
    \varrho \gamma''''
    -  \varrho \langle \gamma'''' , \gamma' \rangle \gamma'
    + \tfrac{3}{2} \varrho \langle \gamma'' , \gamma'' \rangle \gamma''
    + \varrho'' \gamma''
    + \varrho' (2 \gamma'''
    + \tfrac{3}{2} \langle \gamma'' , \gamma'' \rangle \gamma')
    - \Lambda' \gamma'
    - \Lambda \gamma''
    -\mu \gamma' \times \gamma'''
    = 0.
\end{equation*}
\end{theorem}

\begin{proof}
    The conditions of {\cite[Thm. 2.20]{PinkallGross23:DG}} are met, so that for constant \(\mu \in \RR\) we get by Theorems \ref{thm:VariationTotalTorsion} and \ref{thm:TorsionFreeElasticsODE} we obtain
    \begin{equation*}
        \varrho T'''+3 \varrho \langle T'' , T' \rangle T
            +\tfrac{3}{2} \varrho \langle T' , T' \rangle T'
            +\varrho '' T'
            +\varrho ' ( 2 T''+\tfrac{3}{2} \langle T' , T' \rangle T )
            - \Lambda ' T - \Lambda T '
            =
            \mu T \times T''
    \end{equation*}
    which yields the claim after substituting \(-3 \langle T' , T'' \rangle = \langle T''' , T \rangle\).    
\end{proof}

\section{Equivalent Charecterizations}\label{sec:EquivalentCharecterizations}
In this section we derive equivalent characterizations of the elastic curves we derived in the preceding sections. More specifically, we derive analogs of a list of statements which, for the classical case with constant bending-stiffness, relate holonomy constrained elastic curves to dynamical systems such as spinning tops, pendulums, the non-linear Schr\"odinger equation and the vortex-filament flow~\cite{chern_knoppel_pedit_pinkall_2020, PinkallGross23:DG}.

\begin{theorem}\label{thm:3dEquivalences}
For an arc-length parameterized curve \(\gamma \in C^{\infty}([a,b],\RR^3)\) with unit tangent vector \(T\in C^\infty([a,b];S^2)\), the following statements are equivalent: 
    \begin{enumerate}
        \item \(\gamma\) is a holonomy constrained elastic curve.
        \item\label{Equiv1} There is a smooth function \(\Lambda\colon [a,b] \rightarrow \RR\) and constants \(\mu \in \RR\), \textbf{a} \(\in \RR^3\) such that
        \begin{equation*}
            \varrho T'' +\ \tfrac{3}{2} \varrho \langle T' , T' \rangle T
            + \varrho ' T' 
            - \Lambda T
            - \mu T \times T'
            + \textbf{a}
            = 0.
        \end{equation*}
        \item\label{Equiv2} There are constants \(\mu \in \RR\), \(\textbf{a} \in \RR^3\) such that
        \begin{equation*}
            \varrho T'' - \varrho \langle T'' , T \rangle T
            + \varrho ' T'
            + \textbf{a} - \langle \textbf{a} , T \rangle T
            - \mu T \times T'
            = 0.
        \end{equation*}
        \item\label{Equiv3} There are constants \(\mu \in \RR\), \textbf{a}, \textbf{b} \(\in \RR^3\) such that
        \begin{equation*}
            \varrho ( \gamma ' \times \gamma '' ) 
            = -\mu T + \textbf{a}\times \gamma + \textbf{b}.
        \end{equation*}
        \item\label{Equiv4} There are constants \(\textbf{a}, \textbf{b} \in \RR^3\) such that
        \begin{equation*}
            \varrho \gamma ''
            = ( \textbf{a} \times \gamma + \textbf{b} ) \times \gamma'.
        \end{equation*}
    \end{enumerate}
\end{theorem}
\begin{proof}
    For the first equivalence integrate the differential equation from \thmref{thm:ElasticCurvesODE},
    \begin{equation*}
        \varrho T'' + \tfrac{3}{2} \varrho \langle T' , T' \rangle T + \varrho ' T' - \Lambda T - \mu T \times T' = \textbf{a}
    \end{equation*}
    which is precisely \ref{Equiv1}. 

    \textbf{\ref{Equiv1} \(\Leftrightarrow\) \ref{Equiv2}:} 
    Part \ref{Equiv2} follows from \ref{Equiv1} since its the component orthogonal to \(T\):
    \begin{align*}
        0 &= \varrho T'' + \tfrac{3}{2} \varrho \langle T' , T' \rangle T + \varrho ' T' - \Lambda T - \mu T \times T' + \textbf{a} \\
        &\quad - \langle \varrho T'' + \tfrac{3}{2} \varrho \langle T' , T' \rangle T + \varrho ' T' - \Lambda T - \mu T \times T' + \textbf{a} , T \rangle T \\
        &= \varrho T'' - \varrho \langle T'' , T \rangle T + \varrho ' T' + \textbf{a} - \langle \textbf{a} , T \rangle T - \mu T \times T'        
    \end{align*}
For the converse, note that \ref{Equiv2} is the normal component of \ref{Equiv1}, while defining \(\Lambda\) as
    \begin{equation*}
        \Lambda \coloneqq\varrho \langle T'' , T \rangle + \tfrac{3}{2} \varrho \langle T' , T' \rangle + \langle \textbf{a} , T \rangle,
    \end{equation*}
recovers the component parallel to \(T\).
    
\textbf{\ref{Equiv2} \(\Leftrightarrow\) \ref{Equiv3}:}
    First note that \ref{Equiv2} is orhtogonal to \(T\) (as the normal component of \ref{Equiv1}). Then, taking a cross product with \(T\) amounts to a \(90^\circ\)-rotation in the plane normal to \(T\) and 
    \begin{align*}
        0 &= \varrho T'' \times T + \varrho ' T' \times T + \textbf{a} \times T - \mu ( T \times T' ) \times T \\ 
        &= \varrho T'' \times T + \varrho ' T' \times T + \textbf{a} \times T - \mu T' \\
        &= ( \varrho T' \times T + \textbf{a} \times \gamma - \mu T ) '.
    \end{align*}
    Thus, by integration, there is \(\textbf{b}\in\RR^3\) such that
    \begin{equation*}
        -\textbf{b} = \varrho \gamma '' \times \gamma ' + \textbf{a} \times \gamma - \mu \gamma '
    \end{equation*}
    which is equivalent to \ref{Equiv3}. 
    For the converse direction we take the derivative of \ref{Equiv3} and rotate it by \(-90^\circ\) and then take the orthogonal component, which shows equivalence.

    \textbf{\ref{Equiv3} \(\Leftrightarrow\) \ref{Equiv4}:}
    Take the cross product of \ref{Equiv3} with \(T\) to obtain
    \begin{equation*}
        \varrho \gamma '' 
        = ( \textbf{a} \times \gamma + \textbf{b} ) \times \gamma'.
    \end{equation*}
    For the converse direction we must show that the tangential components implied by \ref{Equiv4} are constant, as is required by \ref{Equiv3}. Applying \(\langle\cdot,T\rangle\) to \ref{Equiv3}, we find
    \begin{equation*}
        \langle \rho T\times T', T\rangle = \mu\langle T,T\rangle - \langle \ba\times \gamma + \bb, T\rangle
    \end{equation*}
which is equivalent to 
\begin{equation*}
    \mu = \langle \ba\times\gamma + b, T\rangle. 
\end{equation*}
To verify constancy of this candidate \(\mu\) we check that
\begin{equation*}
    \mu' = \langle \ba\times \gamma', T\rangle + \langle\ba\times\gamma + b,T'\rangle 
    =  \tfrac{1}{\rho}\langle\ba\times\gamma+b,(\ba\times\gamma + b)\times\gamma' \rangle
    = 0.
\end{equation*}
This concludes the proof.
\end{proof}

As an immediate consequence we find another necessary condition for a curve to be holonomy constrained elastic.
\begin{corollary}\label{thm:NecessaryCorollary}
    Let \(\gamma \in C^{\infty}([a,b],\RR^3)\) be an arc-length parameterized holonomy constrained elastic curve with variable bending stiffness \(\varrho\in C^{\infty}([a,b];\RR_{>0})\). Then 
    \[\varrho\,\langle \gamma'\times \gamma'', -\mu\, \gamma' + \gamma \times \textbf{a} + \textbf{b}\rangle >0. \]
\end{corollary}
Although in our setup the condition of \corref{thm:NecessaryCorollary} is not sufficient, Hafner and Bickel~\cite{Hafner23} show that it becomes sufficient when one additionally allows for anisotropic cross-sections.

\subsection{The Curvature of Elastic Curves}\label{sec:TheCurvatureOfElasticCurves}
Equivalent charaterizations of the three types of curves can also be stated in terms of their curvature functions. While traditionally, the curvature function is considered a scalar quantity which is only defined for plane curves, we may also define a \emph{curvature function} \(\kappa\in C^{\infty}([a,b];\RR^{n-1})\) for curves in \(\RR^n\)~\cite[Sec. 4.3]{PinkallGross23:DG}: let \(\gamma\in C^{\infty}([a,b];\RR^n)\) and \(N\coloneqq[N_1,\ldots,N_{n-1}]\) be made of parallel unit normal fields such that \(\det(T,N_1,\ldots, N_{n-1})=1\). From \(1=\langle T, T\rangle\) we notice that \(\langle T,T'\rangle=0\), \ie, \(T'\) is a normal vector field to \(\gamma\). Therefore, we define the curvature function \(\kappa\in C^{\infty}([a,b];\RR^{n-1})\) of \(\gamma\) by 
    \begin{align}
        T' = -N\kappa.
    \end{align}
Note that, by means of \(\kappa\), the derivative of any normal vector field \(Y=Ny\) can be expressed as~\cite{PinkallGross23:DG}
\begin{equation}
    \label{eq:GenericNormalFieldDerivative}
    Y' = \langle\kappa,y\rangle T + Ny'.
\end{equation}
\begin{lemma}\label{thm:TEqsInTermsOfKappa}
Let \(\gamma\in C^{\infty}([a,b];\RR^n)\) be an arc-length parameterized curve, then 
\begin{align}
    \label{eq:TEqsInTermsOfKappa}
    T'&= -N\kappa \notag\\
    T''&= -\langle\kappa,\kappa\rangle T - N\kappa'\\
    T'''&= -3\langle\kappa,\kappa'\rangle T + N(\langle\kappa,\kappa\rangle\kappa - \kappa'').\notag
\end{align}
\end{lemma}

\begin{proof}
    This is straightforward computation for which we use that \(\langle N,N\rangle = \id_{T^\perp}\) implies \(\langle N',N\rangle=0\), hence by \teqref{eq:GenericNormalFieldDerivative} \(N'\kappa = \langle\kappa,\kappa\rangle T\) and \(N'\kappa'=  \langle\kappa',\kappa\rangle T\). Then
    \[T'' = -N'\kappa - N\kappa' = -\langle\kappa,\kappa\rangle T - N\kappa'\]
    and 
    \begin{align*}
        T''' &= -2\langle\kappa',\kappa\rangle T - \langle \kappa,\kappa\rangle T' - N'\kappa' - N\kappa''\\
        &=-3\langle\kappa,\kappa'\rangle T + N(\langle\kappa,\kappa\rangle\kappa - \kappa'').
    \end{align*}
\end{proof}

From plugging the expressions in \teqref{eq:TEqsInTermsOfKappa} into the Euler-Lagrange equations for the different types of elastic curves we obtain equivalent characterizations in terms of conditions on their curvature functions. For example, \teqref{eq:FreeElasticGeneralForT} for the case of constant \(\varrho\) gives 
\begin{align*}
    0 &= T''' + 3\langle T', T''\rangle T + \tfrac{3}{2}\langle T', T'\rangle T' \\
      &= -N(\kappa'' + \tfrac{1}{2}\langle\kappa,\kappa\rangle\kappa).
\end{align*}
We conclude that
    an arc-length parameterized curve \(\gamma\in C^{\infty}([a,b];\RR^n)\) is free elastic if and only if its curvature function satisfies
    \[0=\kappa'' + \tfrac{1}{2}\langle\kappa,\kappa\rangle\kappa.\]
For the special case of planar curves we retrieve the well known formula \cite[Ch. 2]{PinkallGross23:DG}
\begin{align}    \label{eq:ClassicFreeElasticCurvatureEq}
    \kappa'' + \tfrac{\kappa^3}{2} = 0.
\end{align}

In more generality, consider an arc-length parameterized curve \(\gamma\in C^{\infty}([a,b];\RR^n)\) with bending stiffness \(\varrho\in C^{\infty}([a,b];\RR_{>0})\). Then, from  \teqref{eq:FreeElasticGeneralForT} we find that
\begin{align*}
    0 &= \varrho(T''' + 3\langle T', T''\rangle T + \tfrac{3}{2}\langle T', T'\rangle T') + \varrho''T' + 2\varrho'T'' + \tfrac{3}{2}\varrho'\langle T', T'\rangle T \\
    &= N( -\varrho\kappa'' - \tfrac{1}{2}\varrho\langle\kappa,\kappa\rangle\kappa - \varrho''\kappa - 2\varrho'\kappa') + (-2\varrho\langle\kappa,\kappa\rangle + \tfrac{3}{2}\varrho'\langle\kappa,\kappa\rangle)T \\ 
    &= N(-(\varrho\kappa)'' -\tfrac{1}{2}\varrho\langle\kappa,\kappa\rangle\kappa ) + (-\tfrac{1}{2}\varrho'\langle\kappa,\kappa\rangle)T.
\end{align*}
Hence, \(\gamma\) is a free elastic curve with bending stiffness if and only if 
\begin{align*}
    \begin{cases}
        0= (\varrho\kappa)'' +\tfrac{1}{2}\varrho\langle\kappa,\kappa\rangle\kappa\\
        0= \tfrac{1}{2}\varrho'\langle\kappa,\kappa\rangle.
    \end{cases}
\end{align*}
From the second equation we conclude that whenever \(\gamma\) is not a segment of a straight line, it must have constant bending stiffness \(\varrho\) and the defining equations reduce to \teqref{eq:ClassicFreeElasticCurvatureEq}.

By \thmref{thm:TorsionFreeElasticsODE}, adding the arc-length constraint leads to an extra term in the Euler-Lagrange equation for elastic curves which can also be expressed in terms of the curvature function \(\kappa\) as
\begin{align*}
    0 = -\Lambda'T - \Lambda T'  = -\Lambda' T + N(\Lambda\kappa).
\end{align*}
\begin{theorem}\label{thm:TorsionFreeElasticInCurvature}
    An arc-length parameterized curve \(\gamma\in C^{\infty}([a,b];\RR^n)\) with bending stiffness \(\varrho\in C^{\infty}([a,b];\RR_{>0})\) is elastic if and only if its curvature function satisfies
    \begin{align}
        \label{eq:KappaELTorsionFree}
        \begin{cases}
            0 =(\varrho\kappa)'' +\tfrac{1}{2}\varrho\langle\kappa,\kappa\rangle\kappa -\Lambda\kappa\\
            0 = \Lambda' + \tfrac{1}{2}\varrho'\langle\kappa,\kappa\rangle.
        \end{cases}
    \end{align}
    for some \(\Lambda\in C^\infty([a,b])\).
\end{theorem}
Notably, by the second condition: whenever \(\varrho\) is constant, so is \(\Lambda\) and vice versa. 

Moreover, adding a constraint on the holonomy of a curve  \(\gamma\in C^{\infty}([a,b];\RR^3)\) adds the term 
\begin{align*}
    0 = -\mu T\times T'' = -\mu T\times(-\langle\kappa,\kappa\rangle T - N\kappa') = (T\times N)(\mu\kappa') = N(\mu J\kappa'),
\end{align*}
where \(J=\left(\begin{smallmatrix}
    0 & -1 \\ 1 & 0
\end{smallmatrix}\right)\) is the 
endomorphism field on \(T^\perp\) 
which corresponds to a \(90^\circ\) rotation around \(T\) in the normal space. 
\begin{theorem}\label{thm:ElasticInCurvature}
    An arc-length parameterized curve \(\gamma\in C^{\infty}([a,b];\RR^3)\) with bending stiffness \(\varrho\in C^{\infty}([a,b];\RR_{>0})\) is holonomy constrained elastic if and only if its curvature function satisfies
    \begin{align}
        \begin{cases}
            0 =(\varrho\kappa)'' +\tfrac{1}{2}\varrho\langle\kappa,\kappa\rangle\kappa + \mu J\kappa' -\Lambda\kappa \\
            0 = \Lambda' + \tfrac{1}{2}\varrho'\langle\kappa,\kappa\rangle
        \end{cases}
    \end{align}
    for some \(\Lambda\in C^\infty([a,b])\) and \(\mu\in \RR\).
\end{theorem}

\subsection{Closed Examples} 
Preliminary numerical experiments suggest that there are a number of interesting closed examples to be discovered. In \figref{fig:PlanarClosedElastica_nfold_Symmetries} we see approximations of planar examples of closed curves which were obtained from integrating \teqref{eq:PendulumEqDefiningEq}. 
\begin{figure}[t]
   \centering
   \includegraphics[width=.975\columnwidth]{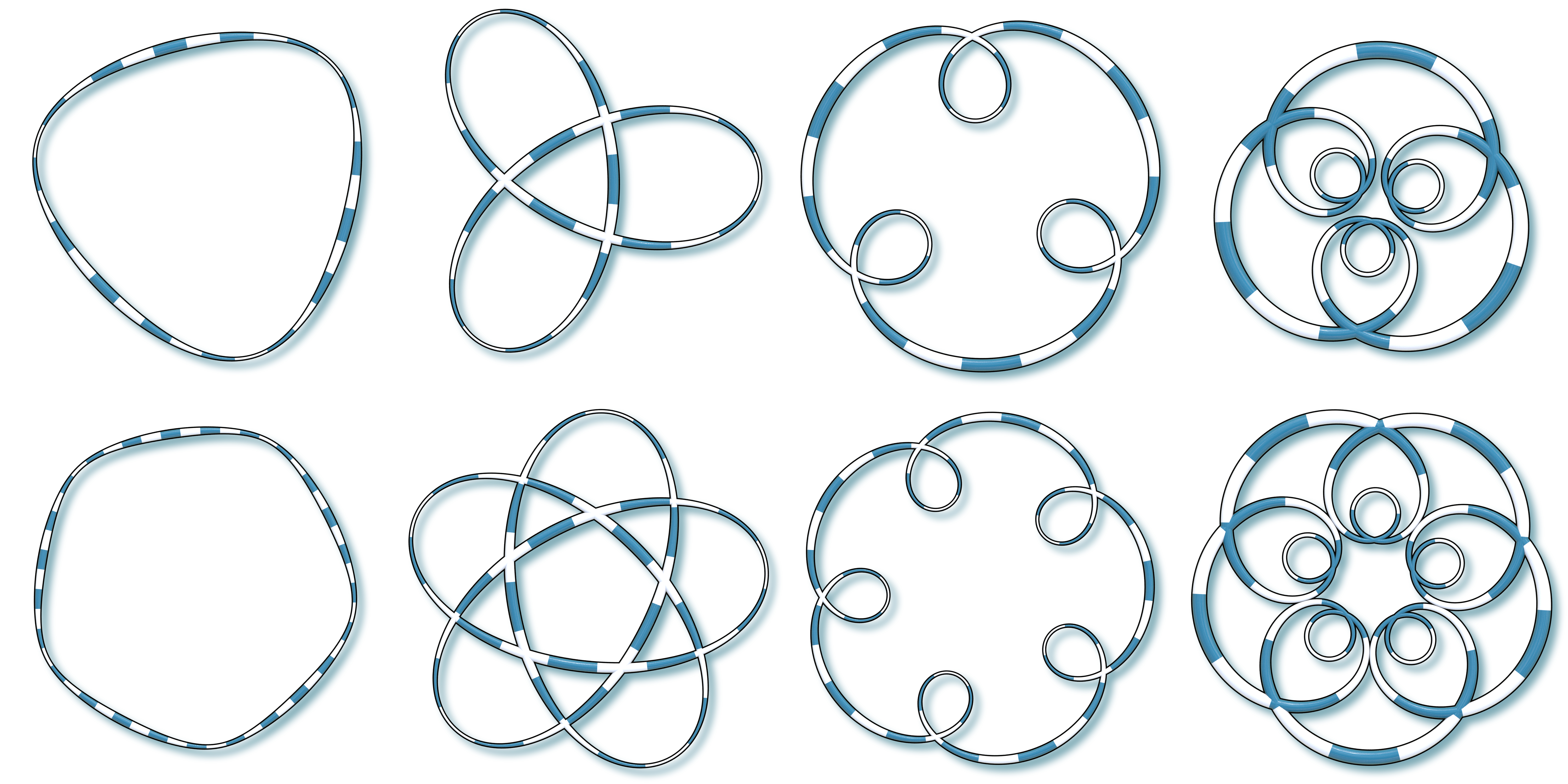}
   \caption{\label{fig:PlanarClosedElastica_nfold_Symmetries}Examples of planar elastic curves with variable bending stiffness obtained from integrating \teqref{eq:PendulumEqDefiningEq} which close up and exhibit \(3\)-fold \resp\@ \(5\)-fold symmetries.
   }
\end{figure}

\begin{figure}[b]
   \centering
   \includegraphics[width=.495\columnwidth]{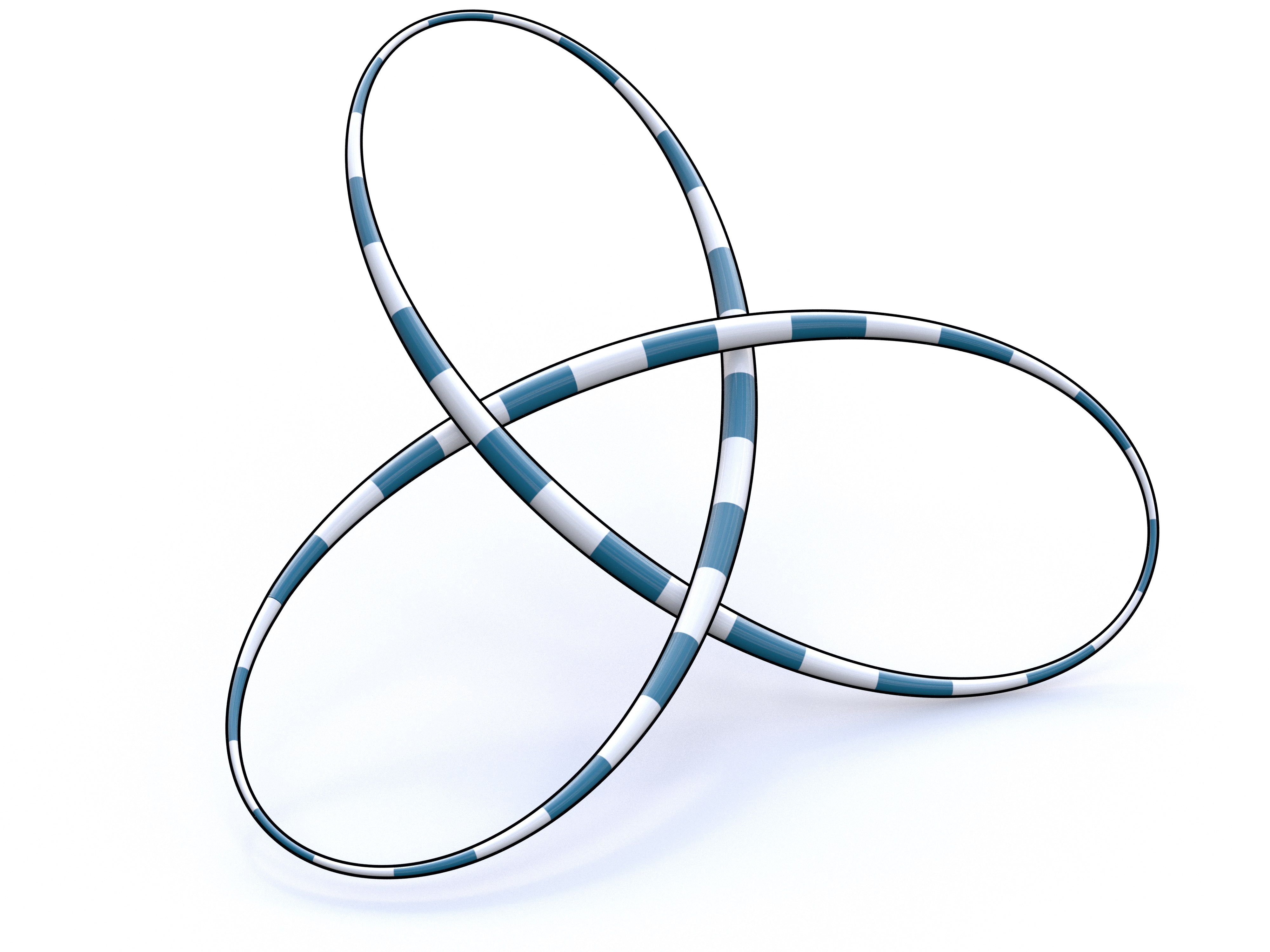}
   \includegraphics[width=.495\columnwidth]{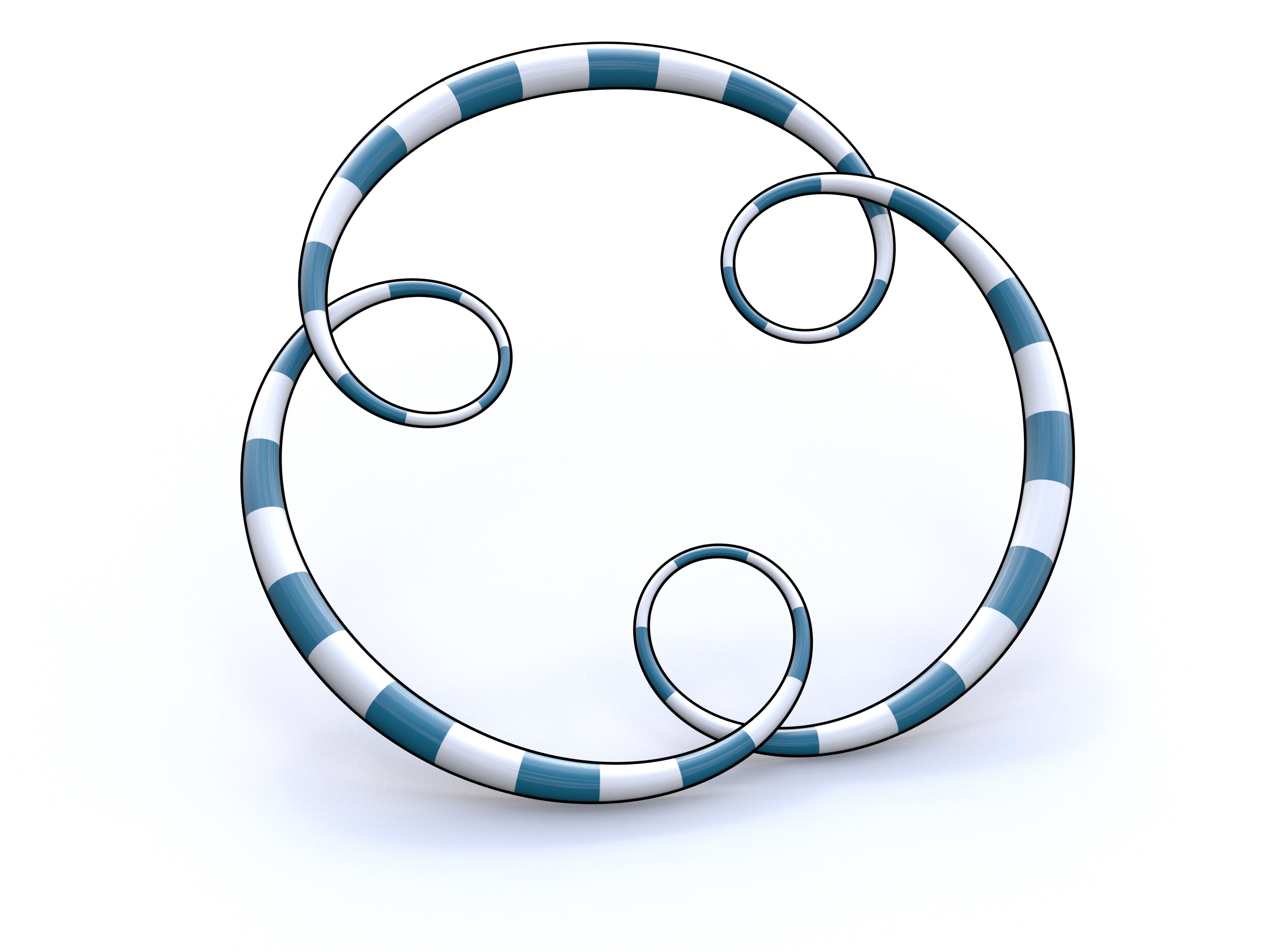}
   \caption{\label{fig:3DElastic_embedded}Embedded examples of closed holonomy constrained elastic curves with variable bending stiffness obtained from integrating condition \ref{Equiv4} in \thmref{thm:3dEquivalences}.
   }
\end{figure}
By adding holonomy, to each of the planar examples, we expect to find a corresponding \(1\)-parameter family of three-dimensional examples~\cite{fuller1971writhing}. Approximations of such non-planar closed elastic curves corresponding to planar examples in \figref{fig:PlanarClosedElastica_nfold_Symmetries} were obtained from integrating condition \ref{Equiv4} in \thmref{thm:3dEquivalences} and are depicted in \figref{fig:3DElastic_embedded}. Note that stationary curves do not generally exhibit any symmetries, and as illustrated by the two examples in \figref{fig:TwoAndThreefoldProteinLike}, even those that appear to do so can display considerable geometric complexity. 
\begin{figure}[h]
   \centering
   \includegraphics[width=\columnwidth]{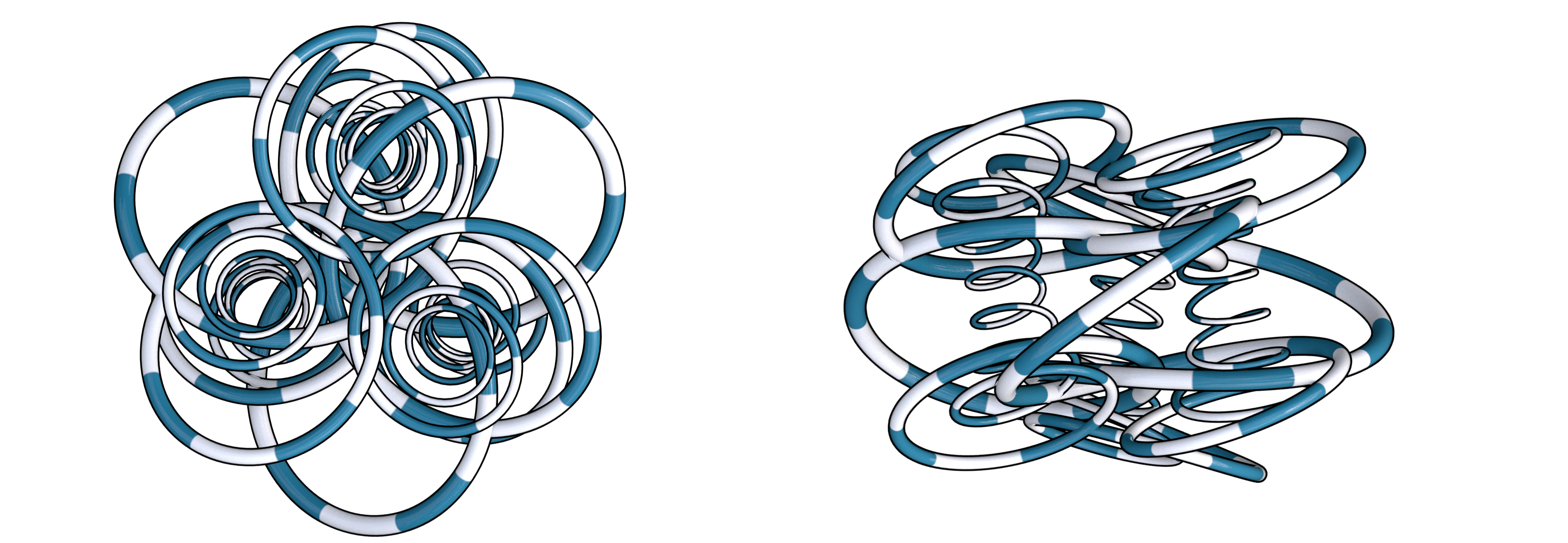}
   \caption{\label{fig:TwoAndThreefoldProteinLike}Top and side view of a closed holonomy constrained elastic curve with variable bending stiffness comprised of three thin helices connected with a threefold symmetry. This example was obtained from integrating condition \ref{Equiv4} in \thmref{thm:3dEquivalences}.
   }
\end{figure}

However, this approach is highly unintuitive, since the relation between those parameters and the constraints is highly nonlinear. Clearly, it would be favorable to have more intuitive computational methods for generating the curves for which the parameters such as length, holonomy and the bending stiffness can be prescribed as an input. However, these computational aspects as well as an explicit analysis of the closing conditions of the curves (see, \eg, \cite{LangerSinger1996}) are beyond the scope of the present work. In \appref{sec:Visualizations}, we outline how our preliminary results were obtained by numerically optimizing the parameters of the ODEs in Theorems \ref{thm:PendulumStiffness} and \ref{thm:3dEquivalences} to satisfy the closing conditions.

\section{Connections with Dynamical Systems}\label{sec:ConnectionsWithDynamicalSystems}
For the case of planar elastic curves with variable bending stiffness Hafner and Bickel found that planar elastic curves with variable bending stiffness are curves that are non-tangentially intersected in only their inflection points (\(\kappa=0\)) by a straight line \[\{x\in\RR^2\mid \langle \ba, x\rangle + c = 0\}\] for some \(\ba\in\RR^2, c\in\RR\).

\begin{theorem}[{\cite{Hafner21}}] \label{thm:HafnerAnalogy}
    A planar, unit-speed curve \(\gamma \in C^{\infty}([a,b],\RR^2) \) is elastic with bending stiffness \(\varrho\) if and only if
    \begin{equation}
        \label{eq:Vector_HafnerAnalogy}
        ( \tfrac{1}{2} \varrho \kappa^2 - \Lambda ) T + ( \varrho \kappa )' JT = \textbf{a}
    \end{equation}
    for some \(\textbf{a} \in \RR^2\). Equivalently, for some \(c \in \RR\), the following equations are satisfied:
    \begin{align}
        \Lambda &= \tfrac{1}{2} \varrho \kappa^2- \langle \textbf{a} , T \rangle,\label{equ:bending_energy_HafnerAnalogy} \\
        \varrho \kappa &= -\langle J \textbf{a} , \gamma \rangle + c.\label{equ:inflection_points_equ}
    \end{align}
\end{theorem}
\begin{proof}
First, integrating the Euler-Lagrange equations in \thmref{thm:TorsionFreeElasticsODE} gives
\begin{equation}
    \label{eq:TorsionFreeElasticELIntegrated}
    \varrho'T' + \varrho T'' + \tfrac{3}{2}\varrho\langle T',T'\rangle T - \Lambda T = \ba 
\end{equation}
for some \(\ba\in\RR^3\). Now, for planes curves, \(N=-JT\), hence \(T' = \kappa JT\) and \(T'' = \kappa'JT - \kappa^2 T\), which plugged into \teqref{eq:TorsionFreeElasticELIntegrated} yields
\teqref{eq:Vector_HafnerAnalogy}. The equivalence statement follows from the observation that \teqref{equ:bending_energy_HafnerAnalogy} is the tangential component of \eqref{eq:Vector_HafnerAnalogy}, while, since \(JT\) denotes the \(90^\circ\)-rotation of \(T\),  \teqref{equ:inflection_points_equ} is its normal component integrated once.
\end{proof}
\begin{figure}[h]
   \centering
   \includegraphics[width=.8\columnwidth]{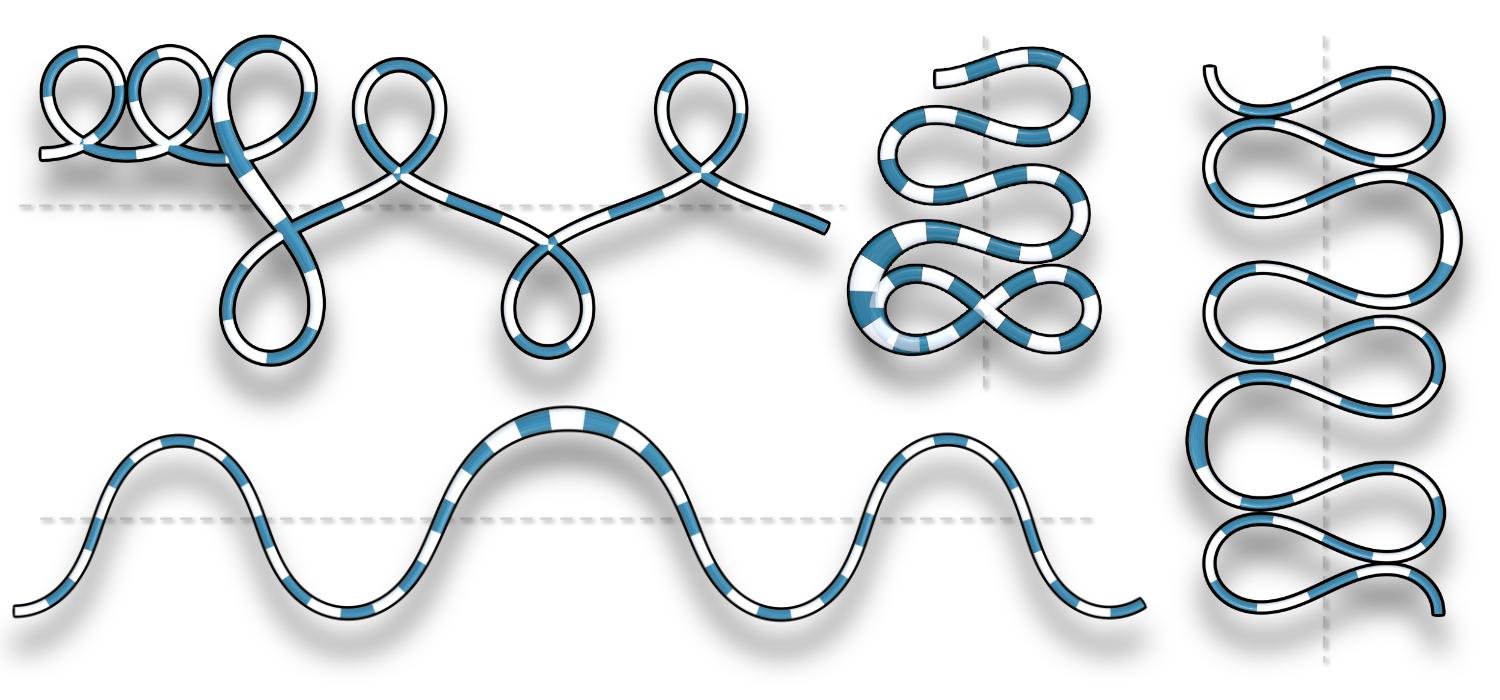}
   \caption{\label{fig:InflectionPointsLine} A collection of planar elastic curves with variable bending stiffness showcasing the result by \cite{Hafner21} that their inflection points are co-linear.
   }
\end{figure}

Notably, \thmref{thm:HafnerAnalogy} generalizes a classical result to the case of variable bending stiffness. It is therefore appropriate to look for other results that can be generalized in this context.

\subsection{Pendulum Analogy}
For the case of constant bending stiffness, arc-length parameterized elastic curves and the pendulum equation share an intricate relationship~\cite{PinkallGross23:DG}. A generalization of this relationship has been established in ~\cite{palmer2020anisotropic} who gave an anisotropic version of the pendulum equation. With \thmref{thm:PendulumStiffness} we give a corresponding generalization of this relationship for the case of variable bending stiffness.

\begin{theorem} \label{thm:PendulumStiffness}
    Let  \(\gamma\in C^{\infty}([a,b];\RR^n)\) be an arc-length parameterized curve with bending stiffness \(\varrho\in C^{\infty}([a,b];\RR_{>0})\). Then \(\gamma\) is elastic if and only if 
    \begin{equation}\label{eq:PendulumEqDefiningEq}
        \varrho T'' + \varrho' T' - \varrho \langle T'' , T \rangle T
        = \textbf{a} - \langle \textbf{a} , T \rangle T
    \end{equation}
    and the pointwise multiplier \(\Lambda\) is given by:
    \begin{equation*}
        \Lambda = \tfrac{1}{2} \varrho \langle T' , T' \rangle
        - \langle \textbf{a} , T \rangle = \tfrac{1}{2} \varrho \langle \kappa , \kappa \rangle
        - \langle \textbf{a} , T \rangle.
    \end{equation*}
    If \(\varrho\) is constant, \(\Lambda\) is constant.
\end{theorem}

\begin{proof}
Similar to the proof of Theorem \ref{thm:3dEquivalences} we obtain
    \begin{equation} \label{equ:ProofPendulum1}
        \varrho T'' + \varrho' T' + \tfrac{3}{2} \varrho \langle T' , T' \rangle T - \Lambda T
        = \textbf{a}
    \end{equation}
    for elastic curves, for some \(\textbf{a} \in \RR^n\).
    \teqref{eq:PendulumEqDefiningEq} is the component orthogonal to \(T\). Taking the inner product of \teqref{equ:ProofPendulum1} with \(T\) yields
    \begin{align*}
        \langle \varrho' T' + \varrho T'' 
        + \tfrac{3}{2} \varrho \langle T' , T' \rangle T
        - \Lambda T , T \rangle
        &= \langle \textbf{a} , T \rangle,
    \end{align*}
    which is equivalent to
    \begin{align*}
        \varrho \langle T'' , T \rangle + \tfrac{3}{2} \varrho \langle T' , T' \rangle - \langle \textbf{a} , T \rangle
        &= \Lambda.
    \end{align*}
    Using \(\langle T'' , T \rangle = - \langle T' , T' \rangle\) results in the second equation. Conversely, let \(T\in C^{\infty}([a,b];\SS^{n-1})\) and \(\varrho\in C^{\infty}([a,b];\RR_{>0})\) satisfy \teqref{eq:PendulumEqDefiningEq} for some \(\textbf{a} \in \RR^n\). Define
    \begin{equation*}
        \Lambda = \tfrac{1}{2} \varrho \langle T' , T' \rangle
        - \langle \textbf{a} , T \rangle
    \end{equation*}
    and substitute \(-\langle T'' , T \rangle\) with \(\langle T' , T' \rangle\) and \(-\langle \textbf{a} , T \rangle\) with \(\Lambda - \tfrac{1}{2} \varrho \langle T' , T' \rangle\) in \teqref{eq:PendulumEqDefiningEq}:
    \begin{equation*}
        \varrho T'' + \varrho' T' + \varrho \langle T' , T' \rangle T = \textbf{a} + \Lambda T - \tfrac{1}{2} \varrho \langle T' , T \rangle T
    \end{equation*}
    which is equivalent to the integrated differential equation for elastic curves with Lagrange multiplier \(\Lambda\). 
\end{proof}

\begin{corollary}\label{thm:CorollaryNewBendingEnergy}
The bending energy of an elastic curve \(\gamma \in C^{\infty}([a,b],\RR^n)\) that is parameterized by arc-length is given by
    \begin{equation*}
        \cB_{\varrho}(\gamma) = \int_a^b \Lambda + \langle \textbf{a} , \gamma(b) - \gamma(a) \rangle
    \end{equation*}
\end{corollary}

\subsubsection{The Planar Case}
Restricting our attention to the \(2\)-dimensional case and defining \(\textbf{a} = -q e_2\) as well as \(T=(\begin{smallmatrix}
        \sin \theta\\
        - \cos \theta
    \end{smallmatrix})\) for \(\theta\colon [a,b] \rightarrow [0,2\pi]\) we find that the first entry of the vector-valued \eqref{eq:PendulumEqDefiningEq} becomes 
\begin{align*}
    q\sin\theta\cos\theta = (\varrho\theta'' - \varrho'\theta')\cos\theta.
\end{align*}
Assuming that \(\theta\notin \tfrac{\pi}{2}\ZZ\) this holds if and only if 
\begin{align}
    \label{eq:DampedPendulumEquation}
    \theta'' = -\tfrac{q}{\varrho}\sin\theta -  \tfrac{\varrho'}{\varrho}\theta'.
\end{align}
Notably, \teqref{eq:DampedPendulumEquation} describes the motion of a pendulum with time–dependent rod length \(\ell(t)\)~\cite[Ch. 8]{Xin:2014:CDA}, since after a change of variables \(\rho = \ell^2\) and \(q = g\ell\), it becomes 
\begin{equation}
    \label{eq:VariableLengthPendulumEquation}
    \theta'' = -\tfrac{g}{\ell}\sin\theta -  2\tfrac{\ell'}{\ell}\theta'.
\end{equation}

Notably, this shows that the relation to the pendulum equation remains intact even when accounting for variable bending stiffness. 

Alterations in the length of a pendulum rod can lead to either amplification or damping of its motion. In particular, well-timed adjustments can transform an initially oscillating pendulum into one that loops fully over the top in continuous circles. In the context of elastic curves, this corresponds to modulating the bending stiffness by superimposing a constant term with one or more smooth bump functions. Such modifications can, for example, induce transitions from periodic curve segments without inflection points---corresponding to looping pendulum motion of the tangent vector---to segments with inflection points---corresponding to oscillatory motion of the tangent vector. \figref{fig:InflectionPointsLine} illustrates examples of this phenomenon.

\subsection{Vortex Filament Flow}\label{sec:VFF}
In this section we will investigate the relations between holonomy constrained elastic curves (with variable bending stiffness) and the dynamics of thin vortex filaments with variable thickness in an incompressible viscous fluid (see, \eg, \cite{Padilla:2019:BRI, chern_knoppel_pedit_pinkall_2020}).
\begin{definition}
    A \emph{vortex filament} is a map \(\gamma\colon S^1\to\RR^3\). Together with an additional function \(a\colon S^1\to \RR_{>0}\) a vortex filament is a \emph{vortex filament with thickness}.
\end{definition}
As with rods of variable bending stiffness, we imagine a vortex filament with thickness to be the geometry swept out by a round disc of radius \(a(s)\), perpendicular to \(T(s)\) and centered at \(\gamma(s)\). Based on a geometric problem formulation, Padilla \etal~\cite{Padilla:2019:BRI} gave first-order equations of motion of these filaments with variable thickness in an incompressible viscous fluid. In the absence of gravity \cite[Eq. (13)]{Padilla:2019:BRI} states that, for \(c_1,c_2\in\RR\) the time evolution of a unit strength vortex-filament is (up to a constant) given by 
\begin{align}
    \label{eq:FiniteThicknessVortexEvolution}
    \dot\gamma = u_{\rm BS}^{(a_0)} - \log(\tfrac{a}{a_0})\, T\times \tfrac{dT}{ds}  -\tfrac{c_2}{a}\tfrac{da}{ds}\,T, 
\end{align}
where 
\[ u_{\rm BS}^{(a_0)} \coloneqq \int_{S^1} \tfrac{T(\tilde s)\times(\gamma(s)-\gamma(\tilde s))}{\left( |\gamma(s)-\gamma(\tilde s)|^2 + c_1a_0^2 \right)^{\nicefrac{3}{2}} }d\tilde s,\]
is the cut-off Biot--Savart integral for a vortex filament of constant thickness \(a_0\) (for which \(0<a\ll a_0\ll 1\)), 
\begin{align*}
    \log(\tfrac{a}{a_0})\, T\times \tfrac{dT}{ds}
\end{align*}
is the \emph{localized induction} term and  the last summand is the \emph{tangential velocity}\footnote{the Lagrangian form of a viscous Burgers' equation} (\cite[Eq. (14)]{Padilla:2019:BRI}).  

\subsubsection{Asymptotic Analysis for Thin Vortex Filaments}
We note that \(\dot\gamma\) becomes infinite in the limit \(a\to 0\), \ie, the vortex filament moves infinitely fast. Therefore, for small filament thickness \(a\), the localized induction term becomes the dominating term for the filament evolution. By a suitable re-scaling of the time, we can control this behavior (slowing down the ``playback speed'' of the filament evolution) and reveal some non-trivial relation to holonomy constrained elastic curves with variable thickness. 

The following spells out this procedure in more detail. Define \(\tau \coloneqq \lambda\,t\) for some \(\lambda\in\RR\). Then, \(\tfrac{\partial}{\partial\tau} = \tfrac{1}{\lambda}\tfrac{\partial}{\partial t}\) and with \(\dot\gamma = \tfrac{\partial}{\partial t}\big\vert_{t=0}\gamma\), we get
\begin{align}
    \label{eq:RetimedVortexFilamentEvolution}
    \tfrac{\partial}{\partial\tau}\big\vert_{\tau=0}\gamma = \tfrac{1}{\lambda}u_{\rm BS}^{(a_0)} - \tfrac{1}{\lambda}\log(\tfrac{a}{a_0})\, T\times \tfrac{dT}{ds}  -\tfrac{1}{\lambda}\tfrac{c_2}{a}\tfrac{da}{ds}\,T.
\end{align}
Now, choose \(\lambda\sim \log(a_1)\) for some \(a_1\) such that \(0<a(s) \sim a_1 \ll a_0\ll 1\). In other words, \(a(s)=a_1^{\varrho(s)}\) for some \(\varrho\colon S^1\to \RR_{>0}\) and 
\begin{equation*}
    \tfrac{\log(a_0)}{\log(a_1)}\eqcolon\varepsilon \ll 1.
\end{equation*}
That is, \teqref{eq:RetimedVortexFilamentEvolution} can be written as
\begin{align}
\label{eq:RetimedVortexFilamentEvolutionWithLambdaChoice}
    \tfrac{\partial}{\partial\tau}\big\vert_{\tau=0}\gamma &= \tfrac{1}{\log(a_1)}u_{\rm BS}^{(a_0)} - \tfrac{\log(a)-\log(a_0)}{\log(a_1)}\, T\times \tfrac{dT}{ds}  -\tfrac{c_2}{\log(a_1)}\tfrac{d\log(a)}{ds}\,T \\
    &= \tfrac{1}{\log(a_1)}u_{\rm BS}^{(a_0)} - (\varrho-\epsilon) T\times \tfrac{dT}{ds}  -c_2\tfrac{d\varrho}{ds}\,T.
\end{align}
Note that \(|u_{\rm BS}^{(a_0)}|\) is of the order \(\log(a_0)\) the first summand is of order \(\varepsilon\). Therefore, for thin vortex filaments with \(\varepsilon\rightarrow 0\), asymptotically (dropping the terms of order \(\varepsilon\)), \teqref{eq:RetimedVortexFilamentEvolutionWithLambdaChoice} is replaced by 
\begin{align}
    \label{eq:AsymptoticFilamentEvolution}
    \tfrac{\partial}{\partial\tau}\big\vert_{\tau=0}\gamma =  -\varrho\, T\times \tfrac{dT}{ds}  -c_2\tfrac{d\varrho}{ds}\,T.
\end{align}
Since the tangential component of \teqref{eq:AsymptoticFilamentEvolution} merely amounts to reparametrizations, so that the evolution of the curve's geometry is determined by the component normal to \(T\) that, after reparametrization by arc-length, agrees with the left-hand side of condition \ref{Equiv3} of \thmref{thm:3dEquivalences}. Therefore, as a corollary of \thmref{thm:3dEquivalences} we conclude:

\begin{corollary}
 Let \(\gamma \in C^\infty(S^1; \mathbb{R}^3)\) be a curve with variable bending stiffness \(\varrho \in C^\infty(S^1; \mathbb{R}_{>0})\). Then \(\gamma\) is holonomy constrained elastic if and only if, when regarded as a thin vortex filament with variable thickness \(a\), its evolution under equation~\eqref{eq:FiniteThicknessVortexEvolution} (modulo reparametrization) asymptotically approaches an infinitesimal  Euclidean motion as \(a \to 0\).
\end{corollary}
However, the motion of the vortex filament does correspond to a global rigid body motion, since the filament’s thickness also evolves over time (see, \eg,~\cite{Padilla:2019:BRI}). Even if the instantaneous evolution corresponds to an infinitesimal rigid body motion, a time-dependent change in thickness modifies the bending stiffness. As a result, unless the thickness is already constant, the curve will not remain elastic at the next time instance, and its geometry will deform. In the special case of constant thickness, we recover a statement analogous to that known for infinitely thin vortex filaments:

\begin{corollary}[{\cite[Cor. 2 (v)]{chern_knoppel_pedit_pinkall_2020}}]
    A curve \(\gamma\in C^\infty(S^1;\RR^3)\) is holonomy constrained elastic if and only if the vortex filament flow \(\dot\gamma = \gamma'\times \gamma''\) evolves, modulo reparametrization, by a Euclidean motion with axis along the monodromy of the initial curve \(\gamma\).
\end{corollary}

\section*{Acknowledgements}
This work was funded in part by the Deutsche Forschungsgemeinschaft (DFG - German Research Foundation) - Project-ID 195170736 - TRR109 ``Discretization in Geometry and Dynamics.'' Additional support was provided through Houdini software, courtesy of SideFX. The authors thank Prof. Albert Chern for helpful discussions on vortex filament dynamics, Dr. Quentin Becker for help with the numerical implementation and the anonymous reviewers for their thoughtful comments improving the manuscript.


\appendix
\section{Notes on Visualizations}
\label{sec:Visualizations}
Throughout the paper we show representative examples encountered during exploration, rather than finely tuned cases. All examples presented in this paper were obtained by integrating ODEs that are equivalent to the Euler--Lagrange equations. Specifically, we use a standard RK4 implementation to solve \teqref{eq:PendulumEqDefiningEq} from \thmref{thm:PendulumStiffness} to obtain elastic curves, as seen in Figs. \ref{fig:curve1_cst_and_sin}, \ref{fig:curve1_agnesi_triple.} and \ref{fig:InflectionPointsLine}, while for holonomy constrained elastic curves we solve condition 5 in \thmref{thm:3dEquivalences}, as seen in \figref{fig:Elastic_HeavyTorsion}. We use either a sinusoidal type of bending stiffness of the form
\begin{equation*}
    \label{eq:SinusoidalBump}
    \varrho(s) = A \sin(s + \xi) + c,
\end{equation*} 
or constant bending stiffness modified by Gaussian bump functions, \ie, 
\begin{equation*}
    \varrho(s)= c + A\exp(-\tfrac{(s-\xi)^2}{2\sigma^2}),
\end{equation*}
for \(A,c,\sigma>0\) and \(\xi\in\RR\). 

For the closed examples, which we deem the most interesting examples, we provide specific parameters. All of the closed examples are obtained by integration of condition \ref{Equiv4} of \thmref{thm:3dEquivalences} and with sinusoidal bending stiffness with \(A=1\) and \(c=1.5\). The initial values are given by \(\gamma(0)=(0,0,0)\) and \(\gamma'(0)=(1,0,0)\). The first row of \figref{fig:PlanarClosedElastica_nfold_Symmetries} 
 is obtained with parameters \(\boldsymbol{a} = 0\), and
\[
\begin{array}{l|c|c|c|c}
 & \text{top left} & \text{top middle left} & \text{top middle right} & \text{top right} \\
\hline
\bb 
& (0,\ 0.37274,\ 0) & (0,\ 0.74556,\ 0) & (0,\ 1.4876,\ 0) & (0,\ 2.6064,\ 0) \\
\xi 
& 4.1314 & 4.6710 & 4.7618 & 4.7360 \\
\end{array}
\]
for the top row, while for the bottom row
\[
\begin{array}{l|c|c|c|c}
 & \text{bottom left} & \text{bottom middle left} & \text{bottom middle right} & \text{bottom right} \\
\hline
\bb 
& (0,\ 0.2234,\ 0) & (0,\ 0.67043,\ 0) & (0,\ 1.3420,\ 0) & (0,\ 2.4584,\ 0) \\
\xi 
& 1.5859 & 1.5859 & 1.5784 & 1.5862 \\
\end{array}
\]
Moreover, the examples shown in \figref{fig:3DElastic_embedded} are obtained with
\[
\begin{array}{l|c|c}
 & \text{left} & \text{right} \\
\hline
\ba & (-1.8502 \times 10^{-6},\ 0,\ 1.2388 \times 10^{-1}) 
  & (-7.4559 \times 10^{-5},\ -6.3287 \times 10^{-13},\ 7.5643 \times 10^{-2}) \\
\bb & (-0.1344,\ 0.1717,\ -0.7235) 
  & (0.0927,\ 0.0081,\ 1.4851) \\
\xi & 27.9713 & 0.1007 \\
\end{array}
\]
while for \figref{fig:TwoAndThreefoldProteinLike}
\[
\begin{array}{l|c}
\ba & (3.8630\times 10^{-2},\  -6.1663\times 10^{-4},\  23.154) 
   \\
\bb & (-2.8675,\  3.6121,\ -9.5421) 
   \\
\xi & 36.3588 \\
\end{array}
\]   

Since a generic choice of parameters will not yield closed curves, we approximate closed curves 
that satisfy the Euler--Lagrange equations, by optimizing the parameter set \(\Theta\)---\(\mathbf{a},\mathbf{b}\in\mathbb{R}^3\), the curve's length \(L\), and shift \(\xi\in\mathbb{R}\) in the bending‐stiffness profile \(\varrho(s)\mapsto\varrho(s+\xi)\)---so that closing conditions hold. 
Concretely, we define an objective function \(J(\Theta)\) encoding the residual at the endpoints after forward-integrating the ODE with a differentiable implementation of an RK4 solver. Using PyTorch’s reverse‐mode automatic differentiation~\cite{Paszke:2019:Pytorch}, we compute gradients \(\tfrac{\partial J}{\partial \Phi}\) and solve  
\[
\argmin_{\Theta} J(\Theta)
\]  
via standard gradient-based optimizers.

\printbibliography

\end{document}